\documentclass[sn-mathphys,Numbered]{sn-jnl}

\usepackage{graphicx}%
\usepackage{multirow}%
\usepackage{amsmath,amssymb,amsfonts}%
\usepackage{amsthm}%
\usepackage{mathrsfs}%
\usepackage[title]{appendix}%
\usepackage{xcolor}%
\usepackage{textcomp}%
\usepackage{manyfoot}%
\usepackage{booktabs}%
\usepackage{algorithmicx}%
\usepackage{algpseudocode}%
\usepackage{listings}%
\usepackage{amsthm,amsmath}
\RequirePackage[numbers]{natbib}
\usepackage[utf8]{inputenc} 

\theoremstyle{thmstyleone}%
\newtheorem{theorem}{Theorem}

\newtheorem{lemma}[theorem]{Lemma}
\theoremstyle{thmstyletwo}%
\newtheorem{example}{Example}%
\newtheorem{remark}{Remark}%
\theoremstyle{thmstylethree}%
\newtheorem{definition}{Definition}%
\newtheorem{assumption}{Assumption}
\usepackage{amssymb}
\usepackage{mathtools}
\usepackage{subcaption}
\usepackage{bbm,bm} 
\usepackage{upgreek,esint}
\usepackage[ruled,noline,linesnumbered]{algorithm2e}

\newcommand{\T}{\mathcal{T}}
\newcommand{\N}{\mathcal{N}}

\newcommand{\eqd}{:=}
\newcommand{\X}{\mathcal{X}}

\newcommand{\xx}{\bm{x}}
\newcommand{\Conv}{\texttt{Conv}}
\newcommand{\uV}{\underline{V}}
\newcommand{\oV}{\overline{V}}
\newcommand{\oP}{\overline{P}}
\newcommand{\uP}{\underline{P}}

\DeclareMathOperator{\blkdiag}{\texttt{blkdiag}}
\DeclareMathOperator{\FB}{\mathsf{T}}
\DeclareMathOperator{\imaginary}{\mathsf{i}}

\newcommand{\pRE}{p^\mathrm{p}} 

\newcommand{\tauC}{\tau^\mathrm{c}}
\newcommand{\rlbRE}{\underline{\rho^\mathrm{p}}} 
\newcommand{\rubRE}{\overline{\rho^\mathrm{p}}} 

\newcommand{\upRE}{\overline{P}^\mathrm{p}} 

\newcommand{\dom}{\mathrm{dom}}

\newcommand{\A}{\mathcal{A}}

\usepackage{acronym}		
\newcommand{\acdefp}[1]{\emph{\aclp{#1}} \textup(\acsp{#1}\textup)\acused{#1}}	

\makeatletter
\newcommand{\pushright}[1]{\ifmeasuring@#1\else\omit\hfill$\displaystyle#1$\fi\ignorespaces}
\newcommand{\pushleft}[1]{\ifmeasuring@#1\else\omit$\displaystyle#1$\hfill\fi\ignorespaces}
\makeatother

\newcommand{\inner}[2]{\langle{#1},{#2}\rangle}
\DeclarePairedDelimiter{\cardinality}{|}{|}
\newcommand{\normx}[2]{\norm{#1}_{#2}}

\newcommand{\filtrationfunction}{\mathcal{F}}
\newcommand{\filtrationk}[1]{\filtrationfunction_{#1}}
\DeclareMathOperator{\produced}{p}
\DeclareMathOperator{\consumed}{c}

\newcommand{\iter}{k} 
 
\newcommand{\Stepsize}{T}

\newcommand{\Stepsizek}[1]{\Stepsize}

\newcommand{\Convex}{\Upsilon}

\newcommand{\Scale}{\Delta}
\newcommand{\Scalek}[1]{\Scale}
\newcommand{\Aggregator}{\mathcal{A}}
\newcommand{\aggregator}{a}

\newcommand{\nbaggregators}{p}
\newcommand{\tim}{t}
\newcommand{\Tim}{\mathcal{T}}
\newcommand{\timlast}{T}
\newcommand{\Node}{\mathcal{N}}
\newcommand{\Nodea}[1]{\Node_{#1}}

\newcommand{\node}{n}
\newcommand{\nodenb}{N}
\newcommand{\nodem}{\node_{-}}
\newcommand{\altnode}{m}
\newcommand{\altnodem}{m_{-}}
\newcommand{\Aall}{A}

\newcommand{\bvector}{b}

\newcommand{\extcostfunction}{\Phi}

\newcommand{\costfunction}{\phi}

\newcommand{\costLAfunction}[1]{\costfunction_{#1}}

\newcommand{\cost}[1]{\costfunction(#1)}
\newcommand{\extcost}[1]{\extcostfunction(#1)}

\newcommand{\costLA}[2]{\costLAfunction{#1}(#2)}

\newcommand{\xall}{\bm{x}}

\newcommand{\xallsol}{\xall^\ast}
\newcommand{\xDSO}{\xall_0}

\newcommand{\xallk}[1]{\xall^{#1}}

\newcommand{\xLAs}{\xall_\Aggregator}

\newcommand{\xLA}[1]{\xall_{#1}}

\newcommand{\XDSO}{\mathcal{X}_0}

\newcommand{\XLA}[1]{\mathcal{X}_{#1}}

\newcommand{\dualvariable}{y}

\newcommand{\dualk}[1]{\dualvariable^{#1}}

\newcommand{\dualvariablecn}[2]{\bm{\dualvariable}^{\textup{#1}}_{#2}}

\newcommand{\dlmqn}[1]{\dualvariablecn{q}{#1}}

\newcommand{\dlmpn}[1]{\dualvariablecn{p}{#1}}

\newcommand{\constraintfunction}{h}

\newcommand{\constraint}[1]{\constraintfunction(#1)}

\newcommand{\cur}{\ell}
\newcommand{\vcur}{\bm{\cur}}
\newcommand{\curn}[1]{\vcur_{#1}}

\newcommand{\curnt}[2]{\cur_{#1,#2}}
\newcommand{\vol}{v}
\newcommand{\vvol}{\bm{\vol}}
\newcommand{\voln}[1]{\vvol_{#1}}

\newcommand{\volnt}[2]{\vol_{#1,#2}}
\newcommand{\bvol}{V}
\newcommand{\ubvol}{\bar{\bvol}}
\newcommand{\ubvoln}[1]{\ubvol_{#1}}
\newcommand{\lbvol}{\underline{\bvol}}
\newcommand{\lbvoln}[1]{\lbvol_{#1}}

\newcommand{\apf}{f}
\newcommand{\vapf}{\bm{\apf}}
\newcommand{\apfn}[1]{\vapf_{#1}}

\newcommand{\apfnt}[2]{\apf_{#1,#2}}
\newcommand{\rpf}{g}
\newcommand{\vrpf}{\bm{\rpf}}
\newcommand{\rpfn}[1]{\vrpf_{#1}}

\newcommand{\rpfnt}[2]{\rpf_{#1,#2}}
\newcommand{\ubpf}{S}
\newcommand{\ubpfn}[1]{\ubpf_{#1}}
\newcommand{\napc}{p}
\newcommand{\vnapc}{\bm{\napc}}
\newcommand{\napcn}[1]{\vnapc_{#1}}

\newcommand{\napcnt}[2]{\napc_{#1,#2}}
\newcommand{\apc}{\napc^{\consumed}}

\newcommand{\apcnt}[2]{\apc_{#1,#2}}
\newcommand{\lbmapc}{E}
\newcommand{\lbmapcn}[1]{\lbmapc_{#1}}
\newcommand{\bapc}{P}
\newcommand{\ubapc}{\overline{\bapc}}
\newcommand{\ubapcnt}[2]{\ubapc_{#1,#2}}
\newcommand{\ubapcn}[1]{\bm{\ubapc}_{#1}}
\newcommand{\lbapc}{\underline{\bapc}}
\newcommand{\lbapcnt}[2]{\lbapc_{#1,#2}}
\newcommand{\lbapcn}[1]{\bm{\lbapc}_{#1}}
\newcommand{\ratiopc}{\tau^{\consumed}}
\newcommand{\ratiopcn}[1]{\ratiopc_{#1}}
\newcommand{\nrpc}{q}
\newcommand{\vnrpc}{\bm{\nrpc}}
\newcommand{\nrpcn}[1]{\vnrpc_{#1}}

\newcommand{\nrpcnt}[2]{\nrpc_{#1,#2}}
\newcommand{\rpc}{\nrpc^{\consumed}}

\newcommand{\rpcnt}[2]{\rpc_{#1,#2}}
\newcommand{\app}{\napc^{\produced}}

\newcommand{\appnt}[2]{\app_{#1,#2}}

\newcommand{\lbratioapp}{\underline{\rho}^{\produced}}
\newcommand{\ubratioapp}{\overline{\rho}^{\produced}}
\newcommand{\lbratioappnt}[2]{\lbratioapp_{#1,#2}}
\newcommand{\ubratioappnt}[2]{\ubratioapp_{#1,#2}}
\newcommand{\rpp}{\nrpc^{\produced}}

\newcommand{\rppnt}[2]{\rpp_{#1,#2}}
\newcommand{\res}{R}
\newcommand{\resn}[1]{\res_{#1}}
\newcommand{\rea}{X}
\newcommand{\rean}[1]{\rea_{#1}}
\newcommand{\con}{G}
\newcommand{\conn}[1]{\con_{#1}}
\newcommand{\sus}{B}
\newcommand{\susn}[1]{\sus_{#1}}

\newcommand{\lagrangianfunction}{L}
\newcommand{\lagrangian}[2]{\lagrangianfunction(#1,#2)}

\DeclarePairedDelimiter{\abs}{\lvert}{\rvert}
\newcommand{\domain}{\texttt{dom}}
\DeclareMathOperator{\prox}{\texttt{prox}}
\DeclareMathOperator{\tr}{Tr}

\DeclarePairedDelimiter{\norm}{\|}{\|}

\newcommand{\Real}{\mathbb{R}}
\newcommand{\symm}{\mathbb{S}}
\newcommand{\R}{\Real}

\newcommand{\Natural}{\mathbb{N}}
\newcommand{\Naturalzero}{\Natural_0}

\newcommand{\Prob}{\mathbb{P}}
\newcommand{\Ex}{\mathbb{E}}
 
\newcommand{\revise}[1]{{\color{black}#1}}
 
\newacro{PDA}{primal-dual algorithm}
\newacro{SOCP}{Second-Order Conic Programming}
\newacro{DSO}{Distribution System Operator}
\newacro{LA}{Load Aggregator}
\newacro{OPF}{Optimal Power Flow}
\newacro{AC}{alternative current}
\newacro{DER}{Distributed Energy Resource}
\newacro{LMP}{Locational Marginal Price}
\newacro{DLMP}{Distribution Locational Marginal Price}
\newacro{OPF}{Optimal Power Flow}
\newacro{ACOPF}{Alternative Current Optimal Power Flow}
\newacro{PPDLMP}{privacy-preserving DLMP solver}
\newacro{ESO}{Expected Separable Overapproximation}



\newcounter{condition}

\usepackage{tikz}
\usepackage{stackengine}
\usetikzlibrary{arrows, positioning}
\usepackage{cleveref}
\usepackage{pbox}
\usepackage{mathtools}

\begin{document}

\title{\revise{Random block coordinate methods for inconsistent convex optimisation problems}}

\author*[1]{\fnm{Mathias}\sur{Staudigl}}\email{mathias.staudigl@uni-mannheim.de}

\author[2]{\fnm{Paulin}\sur{Jacquot}}\email{paulin.jacquot@edf.fr}

\affil*[1]{\orgdiv{Department of Business Informatics and Mathematics}, \orgname{University of Mannheim}, \orgaddress{\street{B6}, \city{Mannheim}, \postcode{68159}, \country{Germany}}}

\affil[2]{\orgdiv{OSIRIS Department}, \orgname{EDF Lab Saclay}, \orgaddress{\street{Bd Gaspard Monge}, \city{Palaisseau}, \postcode{91129}, \country{France}}}

\abstract{We develop a novel randomised block coordinate primal-dual algorithm for a class of non-smooth ill-posed convex programs. Lying in the midway between the celebrated Chambolle-Pock primal-dual algorithm and Tseng's accelerated proximal gradient method, we establish global convergence of the last iterate as well optimal $O(1/k)$ and $O(1/k^{2})$ complexity rates in the convex and strongly convex case, respectively, $k$ being the iteration count. \revise{Motivated by the increased complexity in the control of distribution level electric power systems}, we test the performance of our method on a second-order cone relaxation of an AC-OPF problem. Distributed control is achieved via the distributed locational marginal prices (DLMPs), which are obtained \revise{as} dual variables in our optimisation framework.  }




\maketitle

\section{Introduction}
\label{sec:intro}
In this paper we study non-smooth convex composite optimisation problems of the form 
\begin{equation}\label{eq:P}\tag{P}
\min_{\xx\in\X}\{\Psi(\xx):=\Phi(\xx)+R(\xx)\},
\end{equation}
where $\X:=\arg\min\left\{\frac{1}{2}\norm{\bm{A}\bm{z}-\bm{b}}^{2}\vert\bm{z}\in\R^{m}\right\},$ and the functions $\Phi:\R^{m}\to\R$ and $R:\R^{m}\to(-\infty,\infty]$ are convex and additively decomposable of the form $\Phi(\xx):=\sum_{i=1}^{d}\Phi_{i}(\xx_{i})$ and $R(\xx):=\sum_{i=1}^{d}r_{i}(\xx_{i})$. We assume that each function $\phi_{i}:\R^{m_{i}}\to\R$ is smooth, whereas $r_{i}(\cdot)$ is only required to be proper convex and lower semi-continuous. We typically think of the smooth component $\phi_{i}(\cdot)$ as a convex loss function, whereas $r_{i}(\cdot)$ can take over to role of a regularising or penalty function. In particular, $r_{i}(\cdot)$ can represent an indicator function of a closed convex set $\mathcal{K}_{i}\subset\R^{m_{i}}$, representing individual membership constraints of the decision variable $\xx_{i}\in\R^{m_{i}}$. Thus, problem \eqref{eq:P} can include certain separable block constraints in addition to the non-separable constraints embodied in the set $\X$. This setting gains relevance in distributed optimisation problems with joint coupling constraints, such as multi-agent optimisation problems \cite{facchinei07}. Other examples include convex penalty functions, like the $\ell_{1}$-norm on $\R^{m_{i}}$. The matrix $\bm{A}=[\bm{A}_{1};\ldots;\bm{A}_{d}]$ is decomposed in $q\times m_{i}$ matrices $\bm{A}_{i}\in\R^{q\times m_{i}}$. Accordingly, we use the notation $\xx=[\xx_{1};\ldots;\xx_{d}]$ with each $\xx_{i}\in\R^{m_{i}}$ to represent the blocks of coordinates of $\xx$, and $m=m_{1}+\ldots+m_{d}$. The joint restriction $\xx\in\X$ state that a feasible decision variable is a solution to a linear least squares problem and could be equivalently described as the set of solutions to the normal equations $\X=\{\xx\in\R^{m}\vert \bm{A}^{\top}\bm{A}\xx=\bm{A}^{\top}\bm{b}\}$. When $\bm{b}$ is in the range of $\bm{A}$, then we can solve the linear system $\bm{A}\xx=\bm{b}$ exactly, and the optimisation problem reduces to a linearly constrained convex non-smooth minimisation problem
\begin{equation}\label{eq:Pconsistent}
\min_{\xx}\{\Psi(\xx)=\Phi(\xx)+R(\xx)\} \quad\text{s.t.: } \bm{A}\xx=\bm{b}.
\end{equation}
We call this the \emph{consistent case}. Problems of type \eqref{eq:Pconsistent} are very general and include all generic conic optimisation problems, such as linear programming, second-order cone optimisation and semi-definite programming. In particular, partitioning the matrix $\bm{A}$ into suitably defined blocks, problem \eqref{eq:Pconsistent} is a canonical formulation of various distributed optimisation problems, as the next examples illustrate. 

\begin{example}[Consensus]
Consider the problem 
$$
\min_{\xx\in\R^{p}}\sum_{i=1}^{n}r_{i}(\xx) 
$$
where $r_{1},\ldots,r_{d}$ are closed convex functions on $\R^{p}$. This problem is equivalent to 
$$
\min_{\xx_{1},\ldots,\xx_{d}}\sum_{i=1}^{n}r_{i}(\xx_{i})\quad \text{ s.t.: }\xx_{1}=\xx_{2}=\ldots=\xx_{n}
$$
This can be written as a linear constrained optimisation problem with matrix $\bm{A}$ corresponding to the Laplacian of a connected undirected graph and linear constraint $\bm{A}\xx=0$. 
\end{example}

\begin{example}[Distributed model fitting]
Problem \eqref{eq:P} can also cover composite minimisation problems which canonically arise in machine learning. Consider the problem 
$$
\min_{\bm{u}\in\R^{p}}\ell(\bm{K}\bm{u}-\bm{b})+\rho(\bm{u}),
$$
where $\ell(\cdot)$ is a convex and smooth statistical loss function of the form 
$$
\ell(\bm{K}\bm{u}-\bm{b})=\sum_{i=1}^{q}\ell_{i}(\bm{K}_{i}^{\top}\bm{u}-b_{i})
$$
where $\ell_{i}:\R\to\R$ is the loss for the $i$-th training example, $\bm{K}_{i}\in\R^{p}$ is the feature vector for example $i$, and $b_{i}$ is the output or response for example $i$. Define the variable $\xx=\left[\begin{array}{c}\bm{u}\\ \bm{v}\end{array}\right]\in\R^{p+q}=\R^{m}$ and the linear operator $\bm{A}\xx=\bm{K}\bm{u}-\bm{v}$. Define the functions $\phi_{i}:\R\to\R$ and $r_{i}:\R\to(-\infty,\infty]$ by 
$$
\phi_{i}(\xx_{i}):=\left\{\begin{array}{ll} 
\ell_{i}(v_{i}) & \text{if }p+1\leq i\leq p+d\\ 
0 & \text{ if }1\leq i\leq p 
\end{array}\right.\quad r_{i}(\xx_{i}):=\left\{\begin{array}{ll} 
\rho_{i}(u_{i}) & \text{if }1\leq i\leq p \\ 
0 & \text{ if }p+1\leq i\leq m 
\end{array}\right.
$$
Setting $\Phi(\xx)=\sum_{i=1}^{m}\phi_{i}(\xx_{i})$ and $R(\xx)=\sum_{i=1}^{m}r_{i}(\xx_{i})$ yields the convex optimisation problem 
$$
\min_{\xx}\Phi(\xx)+R(\xx)\quad \text{s.t.: }\bm{A}\xx=\bm{b}.
$$
\end{example}

Problem \eqref{eq:P} is more general than the examples just presented, since we do not insist on $\bm{b}$ belonging to the range of $\bm{A}$. Hence, the optimisation problems of interest to us are non-smooth convex problems with \revise{inconsistent} linear restrictions under which the exact condition $\bm{A}\xx=\bm{b}$ is replaced by the approximate condition $\bm{A}\xx\approx\bm{b}$. This assumption gains relevance in inverse problems and PDE-constrained optimisation, where problems of the from \eqref{eq:P} appear in the method of residuals (Morozov regularisation) \cite{Grasmair2011}. Another instance where ill-posed linear restrictions appear is studied in an application to power systems in Section \ref{sec:application} of this work. Motivated by all these different optimisation problems, this paper derives a unified random \revise{block-coordinate method} which solves problem \eqref{eq:P} under very general assumptions on the problem data.

\subsection{Related Methods}

Our algorithms are closely related to randomised coordinate descent methods, primal–dual coordinate update methods, and accelerated primal–dual methods. In this subsection, let us briefly review these three classes of methods and discuss their relations to our work.

\paragraph{Linear constrained minimisation}
From the theoretical standpoint, one could formulate problem \eqref{eq:P} as a linear constrained optimisation problem of the type \eqref{eq:Pconsistent}, with linear restriction $\bm{A}^{\top}\bm{A}\xx=\bm{A}^{\top}\bm{b}$ (the normal equations). Hence, one could approach problem \eqref{eq:P} via primal-dual techniques directly. While we will show that our main algorithmic scheme (Algorithm \ref{alg:RBCD}) is indeed equivalent to a suitably defined primal-dual method, it can be argued that this connection is not always a recommended solution strategy in practice. First, a primal-dual implementation has to deal with the symmetric matrix $\bm{A}^{\top}\bm{A}$, whose dimension is $m\times m$. If $q$ is much smaller than $m$ then primal-dual methods would have to process many more data points than direct approaches. Second, if $\bm{A}$ is a sparse matrix, then $\bm{A}^{\top}\bm{A}$ is usually not sparse anymore, which leads to heavy use of numerical linear algebra techniques. Finally, it is often not known a-priori whether the linear system is de facto consistent with the given inputs.
 This is in particular the case in the application motivating the development of the numerical scheme to be presented in this paper, which is concerned with the distributed \revise{optimisation} of an electrical distribution system within an AC-optimal power flow framework \cite{Wood13,BienstockCascades}. A basic stability desideratum on a numerical solution method (acting as a decentralised coordination mechanism in our application) is therefore that overall system convergence is guaranteed even if the linear constraints are not satisfied with equality. Our method is exactly achieving this. Section \ref{sec:application} illustrates the performance of our method on a 15-bus AC-OPF problem taken from \cite{papavasiliou18}. 

\paragraph{Primal penalty methods}
An alternative and popular approach to solve \eqref{eq:P} is the penalty method. It consists in solving a sequence of unconstrained optimisation problems 
$
\min_{\xx} \Psi(\xx)+\frac{\sigma_{k}}{2}\norm{\bm{A}\xx-\bm{b}}^{2}, 
$
where $\sigma_{k}$ is a positive and increasing penalty parameter sequence. Intuitively it is clear that, by choosing $\sigma_{k}$ larger, the more importance the optimisation gives to the constraints. Since penalty methods are entirely primal, they do not use duality arguments, and hence they are in principle able to solve the inconsistent case as well. However, their implementation usually involves two loops: an inner loop solving the minimisation problem for a fixed parameter $\sigma_{k}$ to some desired accuracy, followed by an outer loop describing how the penalty parameter is updated. Viewed from this angle, our algorithm is performing these operations in a single-loop fashion. 

\paragraph{Randomised coordinate descent methods}
In the absence of the linear constraints, our algorithm specialises to randomised coordinate descent (RCD), which was first proposed in \cite{NesCoordinate12}, and later generalised to the non-smooth case in \cite{LuXia15,RicTak14}. It was shown that RCD features sublinear rates of convergence with rate $O(1/k)$, $k$ being the iteration counter. Acceleration to a $O(1/k^2)$ complexity and even linear rates for strongly convex problems has been obtained. Extensions to parallel computations, important for large-scale optimization problems, were first proposed in \cite{RicTak16}.

\paragraph{Primal-dual coordinate update methods}
To cope with linear constraints, a very popular approach is the alternating direction method of multipliers (ADMM). Originally, ADMM \cite{GlowMar75,GabMerc76} was proposed for two-block structured problems with separable objective. The convergence and complexity analysis of this method is well documented in the literature \cite{HanSunZha17}; see \cite{EckYao15} for a survey. Direct extensions of the ADMM to multi-block settings such as \eqref{eq:P} is not straightforward, and indeed even may fail to converge \cite{ChenHe16}. Very recently, \cite{Gao:2019aa} proposed a randomised primal–dual coordinate (RPDC) update method, whose asynchronous parallel version was then studied in \cite{Gao:2017aa}. It was shown that RPDC converges with rate $O(1/k)$ under convexity assumption. Improved complexity statementes for multi-block ADMM can be found in \cite{Deng2016}. 

\paragraph{Accelerated primal-dual methods}
It is possible to accelerate the rate of convergence from $O(1/k)$ to $O(1/k^{2})$ for gradient type methods. The first acceleration result was shown by \cite{Nes83} for solving smooth unconstrained problems. The technique has been generalised to accelerated gradient-type methods on possibly non-smooth convex programs \cite{FISTA,Nesterov:2013aa}. Primal–dual methods on solving linearly constrained problems can also be accelerated by similar techniques. Under convexity assumption, the augmented Lagrangian method (ALM) is accelerated from $O(1/k)$ to $O(1/k^{2})$ in \cite{Kang:2015aa}.

\subsection{Contribution}

\paragraph{Methodological contributions}
 We propose a block-coordinate implementation of the method developed by \cite{malitsky2019primaldual} for linearly constrained optimisation, lying midway between the celebrated primal-dual hybrid gradient algorithm \cite{chambolle11} and Tseng's accelerated proximal gradient method \cite{tseng08apgm}. Specifically, our proposed method is a distributed interpretation of the primal-dual algorithm of~\cite{chambolle11} which operates on randomly selected coordinate blocks. A parallel between the Chambolle-Pock method  \cite{chambolle11} and the accelerated proximal gradient of~\cite{tseng08apgm} was already drawn in~\cite{malitsky17,malitsky2019primaldual}. Reducing the primal-dual algorithm to an implementation of Tseng's method enabled~\cite{malitsky17} to derive new convergence results based on primal arguments, thus departing from strong duality requirements and the ergodic rates typically issued for primal-dual methods. Our developments revisit the coordinate-descent implementation proposed in~\cite{luke18} for the basic algorithm of~\cite{malitsky2019primaldual}, to a block-coordinate descent setting featuring Nesterov type of acceleration. In particular, in the strongly convex case, we derive a new step size policy that achieves an accelerated rate of $O(k^{-2})$. Besides the recent preprint \cite{TranLiu21}, we are not aware of a similar algorithm achieving the accelerated convergence rates $O(k^{-2})$ in a fully distributed computational setting. Thus, our main contributions in this paper can be summarised as follows:
\begin{itemize}
\item[(i)] In the convex case, provided problem \eqref{eq:P} possesses a saddle point (defined in Section \ref{sec:Saddle}) our main result is Theorem \ref{thm:convex} which establishes an $O(1/k)$ iteration complexity in terms of the objective function gap and convergence of the last iterate. 
\item[(ii)] In the strongly convex case and uniform sampling of coordinate blocks, our main result is Theorem \ref{thm:sconvex}, which proves an accelerated $O(k^{-2})$ convergence rate in the primal objective function values.
\end{itemize}
We remark that RCD methods have been shown to exhibit linear convergence rates in the strongly convex regime \cite{NesCoordinate12}. Such fast convergence is, however, not to be expected in the presence of linear constraints. While strong convexity of the primal objective ensures smoothness of the Lagrangian dual function, but not its strong concavity. Hence, in general, we do not expect to see linear convergence rates by only assuming strong convexity in the primal. However, we note that \cite{Xu2018} obtain linear convergence rates in the \revise{consistent} regime if there is one block variable that is independent of all others in the objective (but coupled in the linear constraint) and also the corresponding component function is smooth. 
\\
Related to this paper is also the very recent work \cite{Tran-Dinh:2022aa}. They consider a larger class of convex optimisation problems with linear constraints and design a new randomized primal-dual algorithm with single block activation in each iteration and similar complexity results as reported in the present work. However, our method is able to solve \revise{inconsistent} convex optimisation problems with general sampling techniques.

\subsection{Organisation of this paper}
This paper is organised as follows. Section \ref{sec:prelims} describes our block coordinate descent framework. Section \ref{sec:Algorithm} explains in detail our algorithmic approach. Section \ref{sec:convergence} contains all details for the asymptotic convergence and \revise{finite-time} complexity statements. Section \ref{sec:application} describes a challenging application of our algorithm to a distributed optimisation formulation of an AC-OPF problem formulate the distribution grid model, based on the second-order cone relaxation of \cite{farivar13,peng18}. Preliminary numerical results are reported to show the applicability of our method using a 15 bus network studied in \cite{papavasiliou18} as a concrete example.

\section{Preliminaries}
\label{sec:prelims}
\subsection{Notation}
We work in the space $\R^{p}$ composed by column vectors. For $\bm{x},\bm{u}\in\R^{p}$ denote the standard Euclidean inner product $\inner{\bm{x}}{\bm{u}}=\bm{x}^{\top}\bm{u}$ and the Euclidean norm $\norm{\bm{x}}=\inner{\bm{x}}{\bm{x}}^{1/2}$. We let $\symm^{p}_{+}:=\{\bm{B}\in\R^{p\times p}\vert \bm{B}^{\top}=\bm{B},\bm{B}\succeq 0\}$ the space of positive definite matrices. Given $\Lambda\in\symm^{p}_{+}$, we let $\normx{\xall}{\Lambda}:=\inner{\Lambda\xall}{\xall}^{1/2}$ for $\xall\in\R^{p}$. The identity matrix of dimension $p$ is denoted as $\bm{I}_{p}$. Whenever the dimension is clear from the context, we will just write $\bm{I}$. We denote by $\lambda_{\max}(\bm{A})$ the largest eigenvalue of a square $p\times p$ matrix $\bm{A}$. 
We call by $\Gamma_{0}(\R^{p})$ the set of proper convex, lower semi-continuous functions $f:\R^{p}\to (-\infty,+\infty]$. For such a function $f\in\Gamma_{0}(\R^{p})$ the effective domain is defined as $\dom(f):=\{\bm{x}\vert f(\bm{x})<\infty\}$. For $d\in\N$, we set $[d]:=\{1,\ldots,d\}$. For a vector $\nu\in\R^{p}_{++}$, we let $\nu^{-1}$ the vector of reciprocal values. For $\Gamma\in\symm^{m}_{++}$, define the weighted proximal operator 
$\prox^{\Gamma}_{r}(\bm{x}):=\arg\min_{\bm{u}\in\R^{m}}\{r(\bm{u})+\frac{1}{2}\norm{\bm{u}-\bm{x}}_{\Gamma}^{2}\}.$
If $\Gamma:=\blkdiag(\Gamma_{1},\ldots,\Gamma_{d})$, then the proximal operator decomposes accordingly 
$
\prox^{\Gamma}_{r}(\bm{x})=\left(\prox^{\Gamma_{1}}_{r_{1}}(\bm{x}_{1}),\ldots,\prox_{r_{d}}^{\Gamma_{d}}(\bm{x}_{d})\right). 
$
If $f:\R^{m}\to\R$ is differentiable, we denote the gradient of $f$ at $\xx\in\R^{m}$ by $\nabla f(\xx)\in\R^{m}$.

\subsection{Block structure}
We first describe the block setup which has become standard in the analysis of block coordinate methods \cite{RicTak14,fercoq15,RicTak16,Ryu2022}. The block structure of \eqref{eq:P} is given by a decomposition of $\R^{m}$ into $d$ subspaces $\R^{m_{i}},1\leq i\leq d$, so that $\R^{m}=\R^{m_{1}}\times\cdots\times\R^{m_{d}}$. Let $\bm{U}=[\bm{U}_{1},\ldots,\bm{U}_{d}]$ be the $m\times m$ identity matrix, decomposed into column submatrices $\bm{U}_{i}\in\R^{m\times m_{i}}$. For $\xx\in\R^{m}$, let $\xx_{i}=\bm{U}^{\top}_{i}\xx$ be the block of coordinates corresponding to the columns of $\bm{U}_{i}$. Any vector $\bm{s}\in\R^{m}$ can be written as $\bm{s}=\sum_{i=1}^{d}\bm{U}_{i}\bm{s}_{i}$. For $\emptyset\neq I\subseteq[d]$, we write 
$$
\bm{s}_{[I]}=\sum_{i\in I}\bm{U}_{i}\bm{s}_{i}.
$$
%

We denote the $\ell^{2}$-norm on $\R^{m_{i}}$ as $\norm{\cdot}_{i}$. If $\bm{Q}=\blkdiag[\bm{Q}_{1};\ldots,\bm{Q}_{d}]$ is a block-diagonal matrix with $\bm{Q}_{i}\in\symm^{m_{i}}_{++}$, we define a weighted norm on $\R^{m}$ by 
$$
\norm{\bm{s}}_{\bm{Q}}=\sum_{i=1}^{d}\norm{\bm{s}_{i}}_{\bm{Q}_{i}}\qquad\forall \bm{s}\in\R^{m}.
$$

\subsection{Smoothness of $\Phi$}
We assume throughout the paper that $\Phi:\R^{m}\to \R$ is convex and possesses a Lipschitz continuous partial gradient. Specifically, we assume that for each block $i\in[d]$ there exists a matrix $\Lambda_{i}\in\symm^{m_{i}}_{++}$ so that 
\begin{equation}\label{eq:phiconvex}
0 \leq \phi_{i}(\xx_{i}+\bm{t}_{i})-\phi_{i}(\xx_{i})-\inner{\nabla \phi_{i}(\xx_{i})}{\bm{t}_{i}}\leq\frac{1}{2}\norm{\bm{t}_{i}}^{2}_{\Lambda_{i}},
\end{equation}
A typical situation is that $\Lambda_{i}=L_{i}\bm{I}_{m_{i}}$ for a scalar $L_{i}>0$, so that the gradient $\nabla\phi$ is Lipschitz continuous with modulus $L_{i}>0$. Allowing for matrix-valued parameters increases generality and takes into account that norms on the factors $\R^{m_{i}}$ might differ from block to block. Collecting all the matrices $\Lambda_{i}$ in one block diagonal matrix $\Lambda:=\blkdiag[\Lambda_{1};\ldots;\Lambda_{d}]$, it follows 
$$
\Phi(\xx')-\Phi(\xx)-\inner{\nabla\Phi(\xx)}{\xx'-\xx}\leq\frac{1}{2}\norm{\xx'-\xx}^{2}_{\Lambda}.
$$

\subsection{Properties of $R$}
We assume that $R:\R^{m}\to(-\infty,\infty]$ is block separable $R(\xx)=\sum_{i=1}^{d}r_{i}(\xx_{i})$, where the functions $r_{i}:\R^{m_{i}}\to(-\infty,\infty]$ are $\mu_{i}$-strongly convex and closed, with $\mu_{i}\geq 0$. Calling $\Upsilon_{i}=\mu_{i}\bm{I}_{m_{i}}$, this gives 
\begin{equation}\label{eq:rconvex}
r(\xx'_{i})\geq r_{i}(\xx_{i})+\inner{\xi_{i}}{\xx'_{i}-\xx_{i}}+\frac{1}{2}\norm{\xx'_{i}-\xx_{i}}^{2}_{\Upsilon_{i}}\qquad\forall \xx'_{i}\in\R^{m_{i}},\forall\xi\in\partial r_{i}(\xx_{i}).
\end{equation}
Typical examples for the regulariser $r_{i}$ are indicator functions of closed convex sets, i.e. $r_{i}(\xx_{i})=\delta_{\mathcal{K}_{i}}(\xx_{i})$, for $\mathcal{K}_{i}\subset\R^{m_{i}}$ convex and closed, or structure-imposing regularisers like the $L_{p}$-penalty $r_{i}(\xx_{i})=\norm{\xx_{i}}^{p}_{L_{p}(\R^{m_{i}})}$ ($p\geq 1$), prominent in distributed estimation of high-dimensional signals and neural networks. We let $\Convex=\blkdiag[\Convex_{1};\ldots;\Convex_{d}]$ be the $m\times m$ matrix collecting all strong-convexity parameters of the individual blocks. 

\subsection{Quadratic Penalty function}
Define the function 
\begin{equation}\label{eq:h}
h(\xx):=\frac{1}{2}\norm{\bm{A}\xx-b}^{2}.
\end{equation}
Let $h^{*}=\min_{\xx\in\R^{d}} h(\xx)$, so that $\X=\arg\min_{\R^{d}}h(\xx)$. Clearly, $h^{\ast}\geq 0$, with equality if and only if the linear system $\bm{A}\xx=b$ is consistent. 

Since $h$ is quadratic, we have for all $\bm{u},\bm{w},\xx\in\R^{m}$ 
\begin{equation}\label{eq:3pointh}
h(\bm{u})-h(\xx)=\inner{\nabla h(\bm{w})}{\bm{u}-\xx}+\frac{1}{2}\norm{\bm{A}(\bm{u}-\bm{w})}^{2}-\frac{1}{2}\norm{\bm{A}(\xx-\bm{w})}^{2}.
\end{equation}
In particular, if $\xx^{*}\in\X$, the above implies for $\xx=\bm{w}=\xx^{*}$
$$
h(\xx)-h(\xx^{*})=\frac{1}{2}\norm{\bm{A}(\xx-\xx^{*})}^{2}.
$$
We also have 
\begin{align*}
h(\xx+\bm{U}_{i}\bm{t}_{i})&=h(\xx)+\inner{\nabla h(\xx)}{\bm{U}_{i}\bm{t}_{i}}+\frac{1}{2}\norm{\bm{A}_{i}\bm{t}_{i}}^{2}\\
&= h(\xx)+\inner{\bm{U}_{i}^{\top}\nabla h(\xx)}{\bm{t}_{i}}+\frac{1}{2}\norm{\bm{t}_{i}}^{2}_{\bm{A}_{i}^{\top}\bm{A}_{i}}\\
&\leq h(\xx)+\inner{\bm{U}_{i}^{\top}\nabla h(\xx)}{\bm{t}_{i}}+\frac{\lambda_{i}}{2}\norm{\bm{t}_{i}}^{2}_{i}\\
\end{align*}
where $\lambda_{i}\equiv\lambda_{\max}(\bm{A}_{i}^{\top}\bm{A}_{i})\equiv\norm{\bm{A}_{i}}_{2}$, the spectral norm of the matrix $\bm{A}_{i}$, and Lemma \ref{lem:BauCom} (proven in Appendix \ref{app:auxiliary}) immediately implies that for all $t\in[0,1]$
\begin{equation}\label{eq:hstrongconvex}
h(t\xx+(1-t)\xx')=th(\xx)+(1-t)h(\xx')-\frac{t(1-t)}{2}\norm{\bm{A}(\xx-\xx')}^{2}.
\end{equation}

\subsection{On Saddle points}
\label{sec:Saddle}
The optimization problem can be equivalently expressed as the linear constrained optimisation problem 
\begin{equation}
\min_{\xx=(\xx_{1},\ldots,\xx_{d})}\Psi(\xx)=\Phi(\xx)+R(\xx)\quad\text{s.t.: }\bm{A}^{\top}\bm{A}\xx=\bm{A}^{\top}b
\end{equation}
The Lagrangian associated with this non-smooth convex optimization problem is 
$$
\mathcal{L}(\xx,\bm{y})=\Psi(\xx)+\inner{\bm{y}}{\bm{A}^{\top}b-\bm{A}^{\top}\bm{A}\xx}.
$$
\begin{definition}
A pair $(\xx^{\ast},\bm{y}^{\ast})$ is called a \emph{saddle-point} if 
\begin{equation}\label{eq:KKT}
0\in\partial\Psi(\xx^{\ast})-\bm{A}^{\top}\bm{A}\bm{y}^{\ast},\quad \bm{A}^{\top}\bm{A}\xx^{\ast}-\bm{A}^{\top}b=0.
\end{equation}
\end{definition}
For convex programs, the conditions \eqref{eq:KKT} are sufficient for $\xx^{\ast}$ to be a solution of \eqref{eq:P}. They are also necessary if a constraint qualification condition holds (e.g. the Slater condition, stating that there exists $\xx$ in the interior of the domain of $\Psi$ such that $\bm{A}^{\top}\bm{A}\xx=\bm{A}^{\top}b$). 

\subsection{Random Sampling}
\label{sec:sampling}
We next introduce our random sampling strategy. Our approach is very general, and allows for virtually all existing sampling strategies considered in the literature. We refer the reader to \cite{QuRic16I,QuRic16II} for an in-depth systematic overview on this topic. 

Let $(\Omega,\mathcal{F},\Prob)$ be a probability space. By a sampling we mean a random set-valued mapping with values in $2^{[d]}$. We will call the random variable $\mathcal{I}:\Omega\to 2^{[d]}$ a \emph{random sampling}. $\mathcal{I}(\omega)$ defines the selection of blocks employed in a single iteration of our method. The set $\mathcal{I}(\Omega)=\Sigma$ is the set of all possible realisations of the random selection mechanism. Let $\{\iota_{k}\}_{k\in\N}$ represent the stochastic process on $2^{[d]}$ in which each random variable $\iota_{k}$ is an i.i.d copy of $\mathcal{I}$. We refer to such a sampling as an \emph{i.i.d sampling}. We assume that the sampling $\mathcal{I}$ is \emph{proper}: There exists a vector $\bm{\pi}=[\pi_{1},\ldots,\pi_{d}]\in\R^{d}$ with 
$$
\pi_{i}=\Prob(i\in \mathcal{I})\in(0,1)\qquad\forall i\in[d].
$$
With the sampling $\mathcal{I}$, we associate the matrix $\bm{\Pi}\in\R^{d\times d}$ defined as 
$$
(\bm{\Pi})_{ij}:=\Prob(\{i,j\}\subseteq\mathcal{I})\qquad \forall i\neq j, \text{ and }(\bm{\Pi})_{ii}=\pi_{i}.
$$
We note that $\bm{\Pi}\succ 0$ \cite[Thm 3.1]{QuRic16II}. We further define the weighting matrix 
$$
\bm{P}:=\blkdiag[\frac{1}{\pi_{1}}\bm{I}_{m_{1}},\ldots,\frac{1}{\pi_{d}}\bm{I}_{m_{d}}].
$$ 
We emphasise that the random sampling model we adopt here is capable of capturing many stationary randomised activation mechanisms. To illustrate this, consider the following activation mechanisms:
\begin{itemize}
\item \emph{Single coordinate activation:} at each iteration, one coordinate block is activated. This means that $\mathcal{I}(\omega)$ takes values in the discrete set $[d]$ only, i.e. $\Sigma=[d]$. In this case, we necessarily have $\sum_{i=1}^{d}\pi_{i}=1$. 
\item \emph{Uniform Sampling:} for all $i\in[d]$ it holds $\Prob(i\in\mathcal{I})=\Prob(j\in\mathcal{I})$. This implies $\pi_{i}=\frac{\Ex[\abs{\mathcal{I}}]}{d}$ for all $i\in[d]$. A special case of a uniform sampling is the popular class of $m$-nice samplings, arising under the specification where $\Sigma$ is the set of all subsets of $[d]$ containing exactly $m$ elements, each of which is activated with uniform probability. Clearly, in this case one has $\pi_{i}=\frac{m}{d}$ for all $i\in[d]$. 
\item \emph{Full Sampling:} $\Sigma=\{1,\ldots,d\}$, which means that all coordinates are updated in parallel. 
\end{itemize}

\section{\textcolor{blue}{Parallel block coordinate algorithm}}
\label{sec:Algorithm}
\begin{algorithm}[t]
\caption{Distributed Accelerated Proximal Gradient Algorithm}
\label{alg:RBCD}
\DontPrintSemicolon
\SetAlgoNoLine%
\SetKwInOut{Parameters}{{\textbf{Parameters}}}
\SetKwInOut{Init}{{\textbf{Initialization}}}
\SetKwFor{For}{for}{do}{}
\Parameters{ $(\pi_{i})_{i=1}^{d},\bm{Q}^{k}_{i}=\frac{1}{\pi_{i}\tau^{k}}\bm{B}_{i},(\sigma_{k})_{k\geq0},(\tau_{k})_{k\geq 0}$} 
\Init{$\xx^{0}=\bm{w}^{0}\in\R^{m},S_{-1}=0,S_{0}=1$} 
\smallskip 
\SetAlgoLined
%
\For{$k=0,1,2,\dots$}{

$\bm{z}^{k} = \frac{1}{S_{k}}(S_{k-1}\bm{w}^{k}+\sigma_{k}\xx^{k})$\;
%
\text{ {\textbf{draw~$\iota_{k}\subseteq[d]$ as an i.i.d copy of the sampling $\mathcal{I}$.}} } \;
\text{\textbf{Update}}
$$
\xx_{i}^{k+1}=\left\{\begin{array}{ll} 
\FB^{k}_{i}(\xx^{k})\text{ defined in eq.}\eqref{eq:FB} & \text{if }i\in\iota_{k},\\
\xx_{i}^{k} & \text{otherwise.} 
\end{array}\right.
$$
$\bm{w}^{k+1}=\bm{z}^{k}+\frac{\sigma_{k}}{S_{k}}\bm{P}(\xx^{k+1}-\xx^{k})$ \;
%
$S_{k+1}=S_{k}+\sigma_{k+1}$\;
}
 \end{algorithm}
Our random \revise{block-coordinate algorithm} for solving \eqref{eq:P} recursively updates three sequences $\{(\bm{z}^{k},\bm{w}^{k},\xx^{k})\}_{k\geq 0}$. Let $(\sigma_{k})_{k\geq 0}$ be a given sequence of positive numbers. At each iteration $k\geq0$, we define a weight sequence $(S_{k})_{k\geq 0}$ recursively by
$$
S_{0}=1,\quad S_{k}=S_{k-1}+\sigma_{k},\quad\theta_{k}=\frac{\sigma_{k}}{S_{k}}.
$$
Consider the extrapolated point 
\begin{equation}\label{eq:z}
\bm{z}^{k}=(1-\theta_{k})\bm{w}^{k}+\theta_{k}\xx^{k}=\bm{w}^{k}+\theta_{k}(\xx^{k}-\bm{w}^{k}),
\end{equation}
together with the sequence 
\begin{equation}\label{eq:w}
\bm{w}^{k+1}=\bm{z}^{k}+\theta_{k}\bm{P}(\xx^{k+1}-\xx^{k}).
\end{equation}
This reads in coordinates $\bm{w}_{i}^{k+1}=\bm{z}_{i}^{k}+\frac{\theta_{k}}{\pi_{i}}(\xx_{i}^{k+1}-\xx_{i}^{k})$ for all $i\in[d]$. To evaluate $\bm{w}^{k+1}$, we need the primal update $\xx^{k+1}$, which is obtained by a weighted forward-backward step involving the first-order signal 
$
g^{k}_{i}=\nabla\phi_{i}(\xx^{k}_{i})+S_{k}\nabla_{i}h(\bm{z}^{k}).
$
Specifically, given a block-specific scaling matrix $\bm{B}_{i}\in\symm^{m_{i}}_{++}$, we define $\bm{Q}_{i}^{k}:=\frac{1}{\pi_{i}\tau_{k}}\bm{B}_{i}$, and the \emph{weighted forward-backward operator} $\FB^{k}_{i}(\xx^{k})=\prox_{r_{i}}^{\bm{Q}^{k}_{i}}(\xx^{k}_{i}-(\bm{Q}^{k}_{i})^{-1}g^{k}_{i})$, which reads explicitly as 
 \begin{equation}\label{eq:FB}
 \FB^{k}_{i}(\xx^{k}):=\arg\min_{\bm{u}_{i}\in\R^{m_{i}}}\left\{r_{i}(\bm{u}_{i})+\inner{g^{k}_{i}}{\bm{u}_{i}-\xx^{k}_{i}}+\frac{1}{2}\norm{\bm{u}_{i}-\xx^{k}_{i}}^{2}_{\bm{Q}^{k}_{i}}\right\}.
 \end{equation}
We will choose the matrices $\bm{B}_{i}$ later on to adapt for the strong convexity present in the problem data. Putting all these tools together yields Algorithm \ref{alg:RBCD}.

\subsection{Relation to primal-dual methods}
 Our method is very similar to the recent block coordinate primal-dual update \cite{Xu2018}, who focus on the consistent case and uniform samplings \cite{NesCoordinate12,RicTak16,QuRic16II,fercoq15}. We generalise this to arbitrary samplings and show that the sequences produced by Algorithm \ref{alg:RBCD} are equivalent to a primal-dual process in the spirit of \cite{luke18}, formulated as Algorithm \ref{alg:PDA}. 
\begin{algorithm}[t]
\caption{Primal-Dual Block Coordinate Descent Algorithm}
\label{alg:PDA}
\DontPrintSemicolon
\SetAlgoNoLine%
\SetKwInOut{Parameters}{{\textbf{Parameters}}}
\SetKwInOut{Init}{{\textbf{Initialization}}}
\SetKwFor{For}{for}{do}{}
\Parameters{ $(\pi_{i})_{i=1}^{d},\bm{Q}^{k}_{i}=\frac{1}{\pi_{i}\tau^{k}}\bm{B}_{i},(\sigma_{k})_{k\in \Naturalzero }$} 
\Init{$\xx^{0}\in\R^{m},\bm{y}^{0}=\sigma_{0}(\bm{A}\xx^{0}-\bm{b}),\bm{u}^{0}=\bm{A}\xx^{0}-\bm{b}$} 
\smallskip 
\SetAlgoLined
%
\For{$k=0,1,2,\dots$}{
\text{ {\textbf{draw block~$\iota_{k}$ as an i.i.d copy of the sampling $\mathcal{I}$.}} } \;
\text{\textbf{Update}}\; 
$
\xx_{i}^{k+1}=\left\{\begin{array}{ll} 
\prox^{\bm{Q}^{k}_{i}}_{r_{i}}\left(\xx^{k}_{i}-(\bm{Q}^{k}_{i})^{-1}(\nabla\phi_{i}(\xx^{k}_{i})+\bm{A}^{\top}_{i}\bm{y}^{k})\right) &  \text{ if }i\in\iota_{k},\\ 
\xx_{i}^{k} & \text{else.}
\end{array}\right.
$
%
$
\bm{u}^{k+1}=\bm{u}^{k}+\bm{A}(\xx^{k+1}-\xx^{k})\;
$
%

$\bm{y}^{k+1}=\bm{y}^{k}+\sigma_{k}\bm{A}\bm{P}(\xx^{k+1}-\xx^{k})+\sigma_{k+1}\bm{u}^{k+1}$\;
}
 \end{algorithm}
 
Let $\{(\xx^{k},\bm{w}^{k},\bm{z}^{k})\}_{k\geq 0}$ denote the sequences generated by running Algorithm \ref{alg:RBCD}. Let $S_{k}$ be the cumulative step-size process $S_{k}=1+\sum_{t=1}^{k}\sigma_{t}.$ Introduce the sequence $\bm{y}^{k}:=S_{k}(\bm{A}\bm{z}^{k}-\bm{b})$, so that 
$$
S_{k}\nabla_{i}h(\bm{z}^{k})=\bm{A}_{i}^{\top}S_{k}(\bm{A}\bm{z}^{k}-\bm{b})=\bm{A}_{i}^{\top}\bm{y}^{k}.
$$ 
In terms of this new dual variable $\bm{y}^{k}$, we can reorganize the primal update so that 
\begin{align*}
\hat{\bm{x}}^{k+1}_{i}&=\arg\min_{\bm{x}_{i}\in\R^{m_{i}}}\left\{r_{i}(\bm{x}_{i})+\inner{\nabla\phi_{i}(\xx^{k}_{i})+S_{k}\nabla_{i}h(\bm{z}^{k})}{\bm{x}_{i}-\xx^{k}_{i}}+\frac{1}{2}\norm{\bm{x}_{i}-\xx^{k}_{i}}^{2}_{\bm{Q}^{k}_{i}}\right\}.\\
&=\arg\min_{\bm{x}_{i}\in\R^{m_{i}}}\left\{r_{i}(\bm{x}_{i})+\inner{\nabla\phi_{i}(\xx^{k}_{i})+\bm{A}_{i}^{\top}\bm{y}^{k}}{\bm{x}_{i}-\xx^{k}_{i}}+\frac{1}{2}\norm{\bm{x}_{i}-\xx^{k}_{i}}^{2}_{\bm{Q}^{k}_{i}}\right\}\\
&=\prox^{\bm{Q}_{i}^{k}}_{r_{i}}\left(\bm{x}^{k}_{i}-(\bm{Q}^{k}_{i})^{-1}(\nabla\phi_{i}(\xx^{k}_{i})+\bm{A}^{\top}_{i}\bm{y}^{k})\right).
 \end{align*}
The next iterate $\xx^{k+1}$ is obtained by the block-coordinate update rule \eqref{eq:xupdate}. This gives line 3 of Algorithm \ref{alg:PDA}.

Next, observe that
\begin{equation}\label{eq:w-PDA}
\begin{split}
S_{k}(\bm{A}\bm{w}^{k+1}-\bm{b})&=S_{k}(\bm{A}(\bm{z}^{k}+\theta_{k}\bm{P}(\xx^{k+1}-\xx^{k})-\bm{b}))\\
&=\bm{y}^{k}+\theta_{k}S_{k}\bm{A}\bm{P}(\xx^{k+1}-\xx^{k})\\
&=\bm{y}^{k}+\sigma_{k}\bm{A}\bm{P}(\xx^{k+1}-\xx^{k})
\end{split}
\end{equation}
Hence, after introducing the residual $\bm{u}^{k}=\bm{A}\xx^{k}-\bm{b}$, satisfying 
\begin{equation}\label{eq:uactive}
\bm{u}^{k+1}-\bm{u}^{k}=\bm{A}(\xx^{k+1}-\xx^{k})=\sum_{i\in\iota_{k}}\bm{A}_{i}(\hat{\xx}^{k+1}_{i}-\xx^{k}_{i}),
\end{equation}

we obtain line 4 of Algorithm \ref{alg:PDA}, as well as 
\begin{align*}
\bm{y}^{k+1}&=S_{k+1}\left(\bm{A}(\bm{w}^{k+1}+\theta_{k+1}(\xx^{k+1}-\bm{w}^{k+1}))-\bm{b}\right)\\
&=S_{k+1}(\bm{A}\bm{w}^{k+1}-\bm{b})+\sigma_{k+1}(\bm{A}(\xx^{k+1}-\bm{w}^{k+1})\\
&=S_{k}(\bm{A}\bm{w}^{k+1}-\bm{b})+\sigma_{k+1}(\bm{A}\xx^{k+1}-\bm{b})\\
&=\bm{y}^{k}+\sigma_{k}\bm{A}\bm{P}(\xx^{k+1}-\xx^{k})+\sigma_{k+1}\bm{u}^{k+1},
\end{align*}
where we have used in the last step the identity \eqref{eq:w-PDA}. Thereby, we obtain line 5 of Algorithm \ref{alg:PDA}. This completes the verification that the sequence $\{(\xx^{k},\bm{w}^{k},\bm{z}^{k})\}_{k\geq 0}$ generated by Algorithm \ref{alg:RBCD} is equivalent to the just constructed sequence $\{(\xx^{k},\bm{u}^{k},\bm{y}^{k})\}_{k\geq 0}$ corresponding to the iterates of Algorithm \ref{alg:PDA}.

\begin{remark}
The implementation of Algorithm \ref{alg:PDA} is fully parallelisable, involving a computational architecture with $d$ agents and a single central coordinator. The agents manage the coordinate blocks $\xx_{i}$ in a fully decentralised way, using information about the centrally updated dual variable $\bm{y}^{k}$ only. A practical implementation of the computational scheme is as follows:
\begin{enumerate}
\item Given the current data $(\xx^{k},\bm{u}^{k},\bm{y}^{k})$ the coordinator realises a sampling $\iota_{k}$. 
\item All agents in $\iota_{k}$ receive the order to update their control variables in parallel, given their current position $\xx_{i}^{k}$, the data matrix $\bm{A}_{i}$ and the penalty of resource utilisation $\bm{y}^{k}$. 
\item Once all active agents executed their computation, they report to the vector $\bm{A}_{i}(\xx^{k+1}_{i}-\xx^{k}_{i})$ to the central coordinator.
\item The central coordinator updates the penalty parameter $\bm{y}^{k}$ by executing the dual update in line 5 of the Algorithm \ref{alg:PDA}
\end{enumerate}
Distributed primal-dual methods as the one described have received enormous attention in the control and machine learning community; see e.g  \cite{Boyd:2011tj,NecClip13,Bianchi:2015uz,Xu2018,BiaFer19,LatPatr19,Ryu2022}. 
\end{remark}

\begin{remark}
Consider the special case with $\pi_{i}=1/d$ and $d=1$, as well as $\sigma_{k}\equiv\sigma$. Then, Algorithm \ref{alg:PDA} coincides with the primal-dual method of Chambolle-Pock \cite{chambolle11}. In fact, in this case it follows that $\bm{u}^{k}=\bm{A}\xx^{k}-\bm{b}$, and $\bm{y}^{k+1}=\bm{y}^{k}+\sigma(\bm{A}(2\xx^{k+1}-\xx^{k})-b)$. 
\end{remark}

\section{Convergence Analysis}
\label{sec:convergence}
This section is concerned with the convergence properties of Algorithm \ref{alg:RBCD}. We start with a basic descent property of the primal forward-backward step. This will involve the identification of a Lyapunov function to obtain energy decay estimates in a variable metric. Building on this result, we investigate two important scenarios in isolation: First, we consider the merely convex case, which is obtained when $\Upsilon=0$. If $\Upsilon\succ 0$, then accelerated rates in the primal iterates can be obtained. \revise{This, however, requires the derivation of a suitable step-size policy that exploits strong convexity for boosting the performance of the method. Understanding the mechanics of this step-size regime involves a delicate analysis of the thus-constructed step-size policy, which is relegated to Appendix \ref{app:parameters}.}

\subsection{Lyapunov function and key estimates}
We start the analysis with a small extension of ''Property 1" stated in \cite{tseng08apgm} for Bregman proximal gradient algorithms. 
\begin{lemma}\label{lem:xi}
For all $k\geq 0$ and $\xx\in\R^{m}$, define 
\begin{equation}\label{eq:zeta}
\zeta^{k}(\xx)=\Phi(\xx^{k})+\inner{\nabla\Phi(\xx^{k})}{\xx-\xx^{k}}+S_{k}\inner{\nabla h(\bm{z}^{k})}{\xx-\bm{z}^{k}}+\frac{1}{2}\norm{\xx-\xx^{k}}^{2}_{\bm{Q}^{k}}
\end{equation}
Then, for all $\xx\in\R^{m}$ it holds true that 
\begin{equation}\label{eq:ProxInequality}
R(\hat{\xx}^{k+1})+\zeta^{k}(\hat{\xx}^{k+1})\leq R(\xx)+\zeta^{k}(\xx)-\frac{1}{2}\norm{\xx-\hat{\xx}^{k+1}}^{2}_{\bm{Q}^{k}+\Upsilon},
\end{equation}
\end{lemma}
\begin{proof}
Collect the forward-backward operator in coordinates $\FB^{k}(\xx):=[\FB^{k}_{1}(\xx_{1});\ldots;\FB^{k}_{d}(\xx_{d})]$, and set $\hat{\xx}^{k+1}=\FB^{k}(\xx^{k})$. From the definition of the forward-backward operator \eqref{eq:FB}, we see that 
$$
0\in \partial r_{i}(\hat{\xx}^{k+1})+\nabla_{i}\zeta_{k}(\hat{\xx}^{k+1})\qquad\forall i\in[d].
$$
Hence, for all $\xx\in\R^{m}$ and $i\in[d]$, we obtain from eq. \eqref{eq:rconvex}
$$
r_{i}(\xx_{i})\geq r_{i}(\hat{\xx}^{k+1}_{i})+\inner{-\nabla_{i}\zeta^{k}(\hat{\xx}^{k+1})}{\xx_{i}-\hat{\xx}^{k+1}_{i}}+\frac{1}{2}\norm{\xx_{i}-\hat{\xx}^{k+1}_{i}}^{2}_{\Upsilon_{i}}.
$$
Summing over all blocks $i\in[d]$, it follows
$$
R(\xx)\geq R(\hat{\xx}^{k+1})+\inner{-\nabla\zeta^{k}(\hat{\xx}^{k+1})}{\xx-\hat{\xx}^{k+1}}+\frac{1}{2}\norm{\xx-\hat{\xx}^{k+1}}^{2}_{\Upsilon},
$$
Furthermore, it is easy to see that $\xx\mapsto \zeta^{k}(\xx)$ is 1-strongly convex in the norm $\norm{\cdot}_{\bm{Q}^{k}}$. Hence, 
$$
\zeta^{k}(\xx)\geq \zeta^{k}(\hat{\xx}^{k+1})+\inner{\nabla\zeta^{k}(\hat{\xx}^{k+1})}{\xx-\hat{\xx}^{k+1}}+\frac{1}{2}\norm{\xx-\hat{\xx}^{k+1}}^{2}_{\bm{Q}^{k}}.
$$
Adding these two inequalities, and rearranging, gives the claimed result.
\end{proof}

Lemma \ref{lem:Ex1}, together with Lemma \ref{lem:Ex2}, shows that 
 \begin{align*}
\frac{1}{2}\norm{\hat{\xx}^{k+1}-\bm{w}^{k}}^{2}_{\bm{A}^{\top}\bm{A}}&= \frac{1}{2}\norm{\bm{A}(\hat{\xx}^{k+1}-\bm{w}^{k})}^{2}\\
 &\stackrel{\eqref{eq:Ex1}}{=}\frac{1}{2}\Ex_{k}\left[\norm{\bm{A}(\xx^{k}+\bm{P}(\xx^{k+1}-\xx^{k})-\bm{w}^{k})}^{2}\right]\\
&-\frac{1}{2}\Ex_{k}\left[\norm{\bm{A}(\bm{P}\bm{E}_{k}-\bm{I})(\hat{\xx}^{k+1}-\xx^{k})}^{2}\right]\\
&=\frac{1}{2}\Ex_{k}\left[\norm{\bm{A}(\xx^{k}+\bm{P}(\xx^{k+1}-\xx^{k})-\bm{w}^{k})}^{2}\right]-\frac{1}{2}\norm{\hat{\xx}^{k+1}-\xx^{k}}_{\Xi-\bm{A}^{\top}\bm{A}}
\end{align*}
Applying eq. \eqref{eq:centerd}, \eqref{eq:centerd2}, \eqref{eq:hwk} and \eqref{eq:hforward}, as well as the identity $\theta_{k}=\frac{\sigma_{k}}{S_{k}}$, the above becomes 
\begin{equation}\label{eq:energy}
\begin{split}
 \frac{1}{2}\norm{\bm{A}(\hat{\xx}^{k+1}-\bm{w}^{k})}^{2}&\overset{\eqref{eq:hwk}}{=}\frac{S_{k}}{\sigma_{k}}h(\bm{w}^{k})+\frac{S_{k}}{S_{k-1}}\Ex_{k}[h(\xx^{k}+\bm{P}(\xx^{k+1}-\xx^{k}))]\\
 &-\frac{S^{2}_{k}}{\sigma_{k}S_{k-1}}\Ex_{k}[h(\bm{w}^{k+1})]- \frac{1}{2}\norm{\hat{\xx}^{k+1}-\xx^{k}}^{2}_{\Xi-\bm{A}^{\top}\bm{A}}\\
 &\overset{\eqref{eq:hforward},\eqref{eq:centerd2}}{=}\frac{S_{k}}{\sigma_{k}}h(\bm{w}^{k})+\frac{S_{k}}{2S_{k-1}}\Ex_{k}[\norm{\bm{A}(\bm{P}\bm{E}_{k}-\bm{I})(\hat{\xx}^{k+1}-\xx^{k})}^{2}]\\
 &-\frac{S^{2}_{k}}{\sigma_{k}S_{k-1}}\Ex_{k}[h(\bm{w}^{k+1})]- \frac{1}{2}\norm{\hat{\xx}^{k+1}-\xx^{k}}^{2}_{\Xi-\bm{A}^{\top}\bm{A}}+\frac{S_{k}}{S_{k-1}}h(\hat{\xx}^{k+1})\\
& \overset{\eqref{eq:Ex2}}{=}\frac{S_{k}}{\sigma_{k}}h(\bm{w}^{k})+\frac{\sigma_{k}}{2S_{k-1}}\norm{\hat{\xx}^{k+1}-\xx^{k}}^{2}_{\Xi-\bm{A}^{\top}\bm{A}}- \frac{S^{2}_{k}}{\sigma_{k}S_{k-1}}\Ex_{k}[h(\bm{w}^{k+1})]\\
&+\frac{S_{k}}{S_{k-1}}h(\hat{\xx}^{k+1}).
 \end{split}
 \end{equation}
Finally, we need a characterisation of the sequence $(\bm{w}^{k})_{k\geq 0}$, which is quite standard in the analysis of \revise{randomised block-coordinate methods} \cite{fercoq15,QuRic16I,QuRic16II}. We therefore skip the straightforward proof. 
 
 \begin{lemma}\label{lem:average}
 Let $(\xx^{k},\bm{w}^{k})_{k\geq 0}$ be the iterates of Algorithm \ref{alg:RBCD}. Then, for all $k\geq 1$, we have 
 \begin{equation}\label{eq:waverage}
 \bm{w}^{k}_{i}=\sum_{t=0}^{k}\gamma_{i}^{k,t}\xx^{t}_{i},
 \end{equation}
 where for each $i\in[d]$, the coefficients $(\gamma_{i}^{k,t})_{t=0}^{k}$ are defined recursively by setting $\gamma_{i}^{0,0}=1,\gamma_{i}^{1,0}=1-\frac{\theta_{0}}{\pi_{i}},\gamma_{i}^{1,1}=\frac{\theta_{0}}{\pi_{i}}$, and for all $k\geq 1$ 
\begin{equation}\label{eq:gamma}
\gamma^{k+1,t}_{i}:=\left\{\begin{array}{ll} 
(1-\theta_{k})\gamma_{i}^{k,t} & \text{ if }t=0,1,\ldots,k-1,\\
(1-\theta_{k})\gamma_{i}^{k,k}+\theta_{k}(1-\pi_{i}^{-1}) & \text{ if }t=k,\\
\theta_{k}/\pi_{i} & \text{if }t=k+1.
\end{array}
\right.
\end{equation} 
Moreover, for all $k\geq 0$, the following identity holds
\begin{equation}\label{eq:Gammak,k+1}
\gamma_{i}^{k+1,k}+\gamma^{k+1,k+1}_{i}=\theta_{k}+(1-\theta_{k})\gamma_{i}^{k,k}.
\end{equation}
Moreover, if $\theta_{0}\in(0,\min_{i\in[d]}\pi_{i}]$ and $(\theta_{k})_{k\geq 0}$ is a decreasing sequence, then for all $k\geq 0$ and $i\in[d]$, the coefficients $(\gamma_{i}^{k,t})_{t=0}^{k}$ are all positive and add up to 1. 
 \end{lemma}
 Let $\hat{\xx}^{k+1}=\FB_{k}(\xx^{k})$. Using \eqref{eq:phiconvex}, we see 
 \begin{align*}
 \Psi(\hat{\xx}^{k+1})&=\Phi(\hat{\xx}^{k+1})+R(\hat{\xx}^{k+1})\\
 &\leq R(\hat{\xx}^{k+1})+\Phi(\xx^{k})+\inner{\nabla \Phi(\xx^{k})}{\hat{\xx}^{k+1}-\xx^{k}}+\frac{1}{2}\norm{\hat{\xx}^{k+1}-\xx^{k}}^{2}_{\Lambda}
 \end{align*}
 Using \eqref{eq:zeta}, we can continue with the estimate
 \begin{align*}
 \Phi(\xx^{k})+\inner{\nabla\Phi(\xx^{k})}{\hat{\xx}^{k+1}-\xx^{k}}=\zeta^{k}(\hat{\xx}^{k+1})-\frac{1}{2}\norm{\hat{\xx}^{k+1}-\xx^{k}}^{2}_{\bm{Q}^{k}}-S_{k}\inner{\nabla h(\bm{z}^{k})}{\hat{\xx}^{k+1}-\bm{z}^{k}}.
 \end{align*}
 Consequently, for $\xx^{\ast}\in\mathcal{X}=\arg\min_{\xx}h(\xx)$ as a reference point, the following bounds are obtained:
 \begin{align*}
 \Psi(\hat{\xx}^{k+1})&\leq R(\hat{\xx}^{k+1})+\zeta^{k}(\hat{\xx}^{k+1})-S_{k}\inner{\nabla h(\bm{z}^{k})}{\hat{\xx}^{k+1}-\bm{z}^{k}}-\frac{1}{2}\norm{\hat{\xx}^{k+1}-\xx^{k}}^{2}_{\bm{Q}^{k}-\Lambda}\\
 &\overset{\eqref{eq:ProxInequality}}{\leq }R(\xx^{*})+\zeta^{k}(\xx^{*})-\frac{1}{2}\norm{\hat{\xx}^{k+1}-\xx^{*}}^{2}_{\bm{Q}^{k}}-S_{k}\inner{\nabla h(\bm{z}^{k})}{\hat{\xx}^{k+1}-\bm{z}^{k}}\\
 &-\frac{1}{2}\norm{\hat{\xx}^{k+1}-\xx^{k}}^{2}_{\bm{Q}^{k}-\Lambda}\\
 &=\Psi(\xx^{*})-\Phi(\xx^{*})+\zeta^{k}(\xx^{*})-S_{k}\inner{\nabla h(\bm{z}^{k})}{\hat{\xx}^{k+1}-\bm{z}^{k}}\\
 &-\frac{1}{2}\norm{\hat{\xx}^{k+1}-\xx^{*}}^{2}_{\bm{Q}^{k}}-\frac{1}{2}\norm{\hat{\xx}^{k+1}-\xx^{k}}^{2}_{\bm{Q}^{k}-\Lambda}\\
 &\overset{\eqref{eq:zeta}}{=}\Psi(\xx^{*})-\Phi(\xx^{*})+\Phi(\xx^{k})+\inner{\nabla\Phi(\xx^{k})}{\xx^{*}-\xx^{k}}+S_{k}\inner{\nabla h(\bm{z}^{k})}{\xx^{*}-\hat{\xx}^{k+1}}\\
 &+\frac{1}{2}\norm{\xx^{k}-\xx^{*}}^{2}_{\bm{Q}^{k}}-\frac{1}{2}\norm{\hat{\xx}^{k+1}-\xx^{*}}^{2}_{\bm{Q}^{k}+\Upsilon}-\frac{1}{2}\norm{\hat{\xx}^{k+1}-\xx^{k}}^{2}_{\bm{Q}^{k}-\Lambda}\\
  &\overset{\eqref{eq:phiconvex}}{=}\Psi(\xx^{*})+S_{k}\inner{\nabla h(\bm{z}^{k})}{\xx^{*}-\hat{\xx}^{k+1}}+\frac{1}{2}\norm{\xx^{k}-\xx^{*}}^{2}_{\bm{Q}^{k}}-\frac{1}{2}\norm{\hat{\xx}^{k+1}-\xx^{*}}^{2}_{\bm{Q}^{k}+\Upsilon}\\
  &-\frac{1}{2}\norm{\hat{\xx}^{k+1}-\xx^{k}}^{2}_{\bm{Q}^{k}-\Lambda}.
\end{align*}
Via eq. \eqref{eq:hzk}, we obtain 
$$
S_{k}\inner{\nabla h(\bm{z}^{k})}{\xx^{*}-\hat{\xx}^{k+1}}=S_{k-1}\inner{\nabla h(\bm{w}^{k})}{\xx^{*}-\hat{\xx}^{k+1}}+\sigma_{k}\inner{\nabla h(\xx^{k})}{\xx^{*}-\hat{\xx}^{k+1}}.
$$
Next, we apply identity \eqref{eq:3pointh} to each inner product separately, in order to conclude
\begin{align*}
\inner{\nabla h(\bm{w}^{k})}{\xx^{*}-\hat{\xx}^{k+1}}&=h(\xx^{\ast})-h(\hat{\xx}^{k+1})-\frac{1}{2}\norm{\bm{A}(\xx^{*}-\bm{w}^{k})}^{2}+\frac{1}{2}\norm{\bm{A}(\hat{\xx}^{k+1}-\bm{w}^{k})}^{2},\\
\inner{\nabla h(\xx^{k})}{\xx^{*}-\hat{\xx}^{k+1}}&=h(\xx^{\ast})-h(\hat{\xx}^{k+1})-\frac{1}{2}\norm{\bm{A}(\xx^{*}-\xx^{k})}^{2}+\frac{1}{2}\norm{\bm{A}(\hat{\xx}^{k+1}-\xx^{k})}^{2}.
\end{align*}
Combined with the previous display, this shows for $\xx^{\ast}\in\arg\min_{\xx\in\R^{m}}h(\xx)$, 
\begin{align*}
 \Psi(\hat{\xx}^{k+1})-\Psi(\xx^{*})&\leq -S_{k}[h(\hat{\xx}^{k+1})-h(\xx^{*})]-\frac{S_{k-1}}{2}\norm{\bm{A}(\xx^{*}-\bm{w}^{k})}^{2}+\frac{S_{k-1}}{2}\norm{\bm{A}(\hat{\xx}^{k+1}-\bm{w}^{k})}^{2}\\
 &-\frac{\sigma_{k}}{2}\norm{\bm{A}(\xx^{*}-\xx^{k})}^{2}+\frac{\sigma_{k}}{2}\norm{\bm{A}(\hat{\xx}^{k+1}-\xx^{k})}^{2}\\
 &+\frac{1}{2}\norm{\xx^{k}-\xx^{*}}^{2}_{\bm{Q}^{k}}-\frac{1}{2}\norm{\hat{\xx}^{k+1}-\xx^{*}}^{2}_{\bm{Q}^{k}+\Upsilon}-\frac{1}{2}\norm{\hat{\xx}^{k+1}-\xx^{k}}^{2}_{\bm{Q}^{k}-\Lambda}\\
 &=-\left(S_{k}h(\hat{\xx}^{k+1})+S_{k-1}h(\bm{w}^{k})+\sigma_{k}h(\xx^{k})-2S_{k}h(\xx^{*})\right)+\frac{S_{k-1}}{2}\norm{\bm{A}(\hat{\xx}^{k+1}-\bm{w}^{k})}^{2}\\
&+\frac{1}{2}\norm{\xx^{k}-\xx^{*}}^{2}_{\bm{Q}^{k}}-\frac{1}{2}\norm{\hat{\xx}^{k+1}-\xx^{*}}^{2}_{\bm{Q}^{k}+\Upsilon}-\frac{1}{2}\norm{\hat{\xx}^{k+1}-\xx^{k}}^{2}_{\bm{Q}^{k}-\Lambda-\sigma_{k}\bm{A}^{\top}\bm{A}}.
\end{align*}
Finally, applying identity \eqref{eq:energy}, one sees 
\begin{align*}
&-\left(S_{k}h(\hat{\xx}^{k+1})+S_{k-1}h(\bm{w}^{k})+\sigma_{k}h(\xx^{k})-2S_{k}h(\xx^{*})\right)+\frac{S_{k-1}}{2}\norm{\bm{A}(\hat{\xx}^{k+1}-\bm{w}^{k})}^{2}\\
&=-\left(S_{k}h(\hat{\xx}^{k+1})+S_{k-1}h(\bm{w}^{k})+\sigma_{k}h(\xx^{k})-2S_{k}h(\xx^{*})\right)\\
&+S_{k-1}\left(\frac{S_{k}}{\sigma_{k}}h(\bm{w}^{k})+\frac{\sigma_{k}}{2S_{k-1}}\norm{\hat{\xx}^{k+1}-\xx^{k}}^{2}_{\Xi-\bm{A}^{\top}\bm{A}}- \frac{S^{2}_{k}}{\sigma_{k}S_{k-1}}\Ex_{k}[h(\bm{w}^{k+1})]+\frac{S_{k}}{S_{k-1}}h(\hat{\xx}^{k+1})\right)\\
&=\frac{S^{2}_{k-1}}{\sigma_{k}}h(\bm{w}^{k})-\frac{S^{2}_{k}}{\sigma_{k}}\Ex_{k}[h(\bm{w}^{k+1})]+2S_{k}h(\xx^{*})-\sigma_{k}h(\xx^{k})+\frac{\sigma_{k}}{2}\norm{\hat{\xx}^{k+1}-\xx^{k}}^{2}_{\Xi-\bm{A}^{\top}\bm{A}}\\
&=\frac{S^{2}_{k-1}}{\sigma_{k}}(h(\bm{w}^{k})-h(\xx^{*}))-\frac{S^{2}_{k}}{\sigma_{k}}\Ex_{k}[h(\bm{w}^{k+1})-h(\xx^{*})]-\sigma_{k}(h(\xx^{k})-h(\xx^{*}))+\frac{\sigma_{k}}{2}\norm{\hat{\xx}^{k+1}-\xx^{k}}^{2}_{\Xi-\bm{A}^{\top}\bm{A}},
\end{align*}
where the last step uses the identity $2S_{k}=\frac{S^{2}_{k}}{\sigma_{k}}-\frac{S^{2}_{k-1}}{\sigma_{k}}+\sigma_{k}$. Plugging this expression into the penultimate display, we arrive that 
\begin{equation}\label{eq:finalpsi}
\begin{split}
 \Psi(\hat{\xx}^{k+1})-\Psi(\xx^{*})&\leq \frac{S^{2}_{k-1}}{\sigma_{k}}(h(\bm{w}^{k})-h(\xx^{*}))-\frac{S^{2}_{k}}{\sigma_{k}}\Ex_{k}[h(\bm{w}^{k+1})-h(\xx^{*})]-\sigma_{k}(h(\xx^{k})-h(\xx^{*}))\\
 &+\frac{1}{2}\norm{\xx^{k}-\xx^{*}}^{2}_{\bm{Q}^{k}}-\frac{1}{2}\norm{\hat{\xx}^{k+1}-\xx^{*}}^{2}_{\bm{Q}^{k}+\Upsilon}-\frac{1}{2}\norm{\hat{\xx}^{k+1}-\xx^{k}}^{2}_{\bm{Q}^{k}-\Lambda-\sigma_{k}\Xi}.
 \end{split}
 \end{equation}
 Define the vector-valued function $\vec{\psi}(\xx):=[\phi_{1}(\xx_{1})+r_{1}(\xx_{1}),\ldots,\phi_{d}(\xx_{d})+r_{d}(\xx_{d})]^{\top}\in\R^{d}$. Let 
$$
\hat{\Psi}_{k}:=\sum_{t=0}^{k}\sum_{i=1}^{d}\gamma_{i}^{k,t}\psi_{i}(\xx^{t}_{i})=\sum_{t=0}^{k}\gamma^{k,t}\cdot \vec{\psi}(\xx^{t}),
$$
 where $\gamma^{k,t}\cdot \vec{\psi}(\xx^{t}):=\sum_{i=1}^{d}\gamma_{i}^{k,t}\psi_{i}(\xx^{t}_{i})$. Thanks to Lemma \ref{lem:average}, we have $\hat{\Psi}_{k}\geq \Psi(\bm{w}^{k})$ for all $k\geq 0$. 
 \begin{lemma}
 Let $\bm{M}:=\blkdiag[\bm{M}_{1};\ldots;\bm{M}_{d}]$ with $\bm{M}_{i}\in\symm^{m_{i}}_{++}$. We have 
 \begin{align}
 \Ex_{k}[\norm{\xx^{k+1}-\xx^{*}}^{2}_{\bm{P}\bm{M}}]&=\norm{\hat{\xx}^{k+1}-\xx^{*}}^{2}_{\bm{M}}+\norm{\xx^{k}-\xx^{*}}^{2}_{(\bm{P}-\bm{I})\bm{M}} \label{eq:normx},\\
 \Ex_{k}[\hat{\Psi}_{k+1}]&=(1-\theta_{k})\hat{\Psi}_{k}+\theta_{k}\Psi(\hat{\xx}^{k+1}).\label{eq:psi}
 \end{align}
 \end{lemma}
 \begin{proof}
Eq. \eqref{eq:normx} can be easily seen by observing 
\begin{align*}
\Ex_{k}[\norm{\xx^{k+1}_{i}-\xx^{*}_{i}}^{2}_{\frac{1}{\pi_{i}}\bm{M}_{i}}]&=\pi_{i}\norm{\hat{\xx}^{k+1}_{i}-\xx^{*}_{i}}^{2}_{\frac{1}{\pi_{i}}\bm{M}_{i}}+(1-\pi_{i})\norm{\xx^{k}_{i}-\xx^{\ast}_{i}}^{2}_{\frac{1}{\pi_{i}}\bm{M}_{i}}\\
&=\norm{\hat{\xx}^{k+1}_{i}-\xx^{*}_{i}}^{2}_{\bm{M}_{i}}+(\pi^{-1}_{i}-1)\norm{\xx^{k}_{i}-\xx^{*}_{i}}^{2}_{\bm{M}_{i}}.
\end{align*}
Summing over all $i\in\{1,\ldots,d\}$ gives the result. To prove \eqref{eq:psi}, we first observe 
$$
\Ex_{k}[\psi_{i}(\xx^{k+1}_{i})]=\pi_{i}\psi_{i}(\hat{\xx}^{k+1}_{i})+(1-\pi_{i})\psi_{i}(\xx^{k}_{i}),
$$
so that $\Ex_{k}[\vec{\psi}(\xx^{k+1})]=\bm{P}^{-1}\vec{\psi}(\hat{\xx}^{k+1})+(\bm{I}-\bm{P}^{-1})\vec{\psi}(\xx^{k})$. It follows 
\begin{align*}
\hat{\Psi}_{k+1}&=\sum_{t=0}^{k-1}\gamma^{k+1,t}\cdot \vec{\psi}(\xx^{t})+\gamma^{k+1,k}\cdot\vec{\psi}(\xx^{k})+\gamma^{k+1,k+1}\cdot\vec{\psi}(\xx^{k+1})\\
&=\sum_{t=0}^{k-1}(1-\theta_{k})\gamma^{k,t}\cdot\vec{\psi}(\xx^{t})+\left((1-\theta_{k})\gamma^{k,k}+\theta_{k}(1-\pi^{-1})\right)\cdot\vec{r}(\xx^{k})+\theta_{k}\pi^{-1}\cdot\vec{\psi}(\xx^{k+1})\\
&=(1-\theta_{k})\hat{\Psi}_{k}+\theta_{k}\left((1-\pi^{-1})\cdot\vec{\psi}(\xx^{k})+\pi^{-1}\cdot\vec{\psi}(\xx^{k+1})\right).
\end{align*}
Taking conditional expectations on both sides, it follows 
\begin{align*}
\Ex_{k}[\hat{\Psi}_{k+1}]&=(1-\theta_{k})\hat{\Psi}_{k}+\theta_{k}\left((1-\pi^{-1})\cdot\vec{\psi}(\xx^{k})+\Psi(\hat{\xx}^{k+1})+(\pi^{-1}-1)\cdot\vec{\psi}(\xx^{k})\right)\\
&=(1-\theta_{k})\hat{\Psi}_{k}+\theta_{k}\Psi(\hat{\xx}^{k+1}).
\end{align*}
\end{proof}
The next result characterises a suitable Lyapunov function in an adapted variable metric. We set $\bm{B}:=\blkdiag[\bm{B}_{1};\ldots;\bm{B}_{d}]$. 
\begin{lemma}\label{lem:lyapunovRS}
Assume that $\theta_{0}\in(0,\min_{i\in[d]}\pi_{i}]$, and that the sequences $(\sigma_{k})_{k\geq 0},(\tau_{k})_{k\geq 0}$ satisfy the matrix inequalities
 \begin{align}\label{eq:MI1}
&\bm{P}\bm{B}\succeq\tau_{k}(\Lambda+\sigma_{k}\Xi),\\
&\bm{P}^{2}\bm{B}+\tau_{k}\bm{P}\Upsilon\succeq\frac{\tau_{k}\sigma_{k+1}}{\tau_{k+1}\sigma_{k}}(\bm{P}^{2}\bm{B}-\tau_{k+1}(\bm{I}-\bm{P})\Upsilon).
\label{eq:MI2}
\end{align}
Define the matrix-valued sequence $(\bm{W}_{k})_{k\geq 0}\subset\symm^{m}_{++}$ by $\bm{W}_{k}=\frac{\sigma_{k}}{\tau_{k}}\bm{P}^{2}\bm{B}+\sigma_{k}(\bm{P}-\bm{I})\Upsilon$, and introduce the functions 
 \begin{align}
 &F_{k}:=\hat{\Psi}_{k}+S_{k-1}(h(\bm{w}^{k})-h(\xx^{*})),\label{eq:Fk}\text{ and }\\
 &V_{k}(\xx):=\frac{1}{2}\norm{\xx^{k}-\xx}^{2}_{\bm{W}_{k}}+S_{k-1}(F_{k}-\Psi(\xx)),
\label{eq:Vk}
 \end{align}
We have for all $\xx^{*}\in\X$, 
\begin{equation}\label{eq:martingale}
 \Ex_{k}\left[V_{k+1}(\xx^{*})\right]\leq V_{k}(\xx^{*})-\sigma^{2}_{k}(h(\xx^{k})-h(\xx^{*}))-\frac{1}{2}\norm{\hat{\xx}^{k+1}-\xx^{k}}^{2}_{\frac{\sigma_{k}}{\tau_{k}}(\bm{P}\bm{B}-\tau_{k}(\sigma_{k}\Xi+\Lambda))}.
\end{equation}
\end{lemma}
\begin{proof}
Using identity \eqref{eq:normx} with $\bm{M}=\bm{Q}_{k}+\Upsilon$, we get 
 \begin{align*}
 \Ex_{k}\left[\frac{1}{2}\norm{\xx^{k+1}-\xx^{*}}^{2}_{\bm{P}\bm{Q}^{k}+\bm{P}\Upsilon}\right]=\frac{1}{2}\norm{\hat{\xx}^{k+1}-\xx^{*}}^{2}_{\bm{Q}^{k}+\Upsilon}+\frac{1}{2}\norm{\xx^{k}-\xx^{*}}^{2}_{(\bm{P}-\bm{I})(\bm{Q}^{k}+\Upsilon)}
 \end{align*}
 Furthermore, eq. \eqref{eq:psi} yields
 \begin{align*}
 \frac{S_{k}}{\sigma_{k}}\Ex_{k}[\hat{\Psi}_{k+1}-\Psi(\xx^{*})]=\frac{S_{k-1}}{\sigma_{k}}\left(\hat{\Psi}_{k}-\Psi(\xx^{*})\right)+\left(\Psi(\hat{\xx}^{k+1})-\Psi(\xx^{*})\right),
 \end{align*}
 and \eqref{eq:finalpsi} shows 
 \begin{align*}
\frac{S^{2}_{k}}{\sigma_{k}} \Ex_{k}[h(\bm{w}^{k+1})-h(\xx^{*})]&\leq \frac{S^{2}_{k-1}}{\sigma_{k}}[h(\bm{w}^{k})-h(\xx^{*})]-[\Psi(\hat{\xx}^{k+1})-\Psi(\xx^{*})]-\sigma_{k}[h(\xx^{k})-h(\xx^{*})]\\
&+\frac{1}{2}\norm{\xx^{k}-\xx^{*}}^{2}_{\bm{Q}^{k}}-\frac{1}{2}\norm{\hat{\xx}^{k+1}-\xx^{*}}^{2}_{\bm{Q}^{k}+\Upsilon}-\frac{1}{2}\norm{\hat{\xx}^{k+1}-\xx^{k}}^{2}_{\bm{Q}^{k}-\Lambda-\sigma_{k}\Xi}.
 \end{align*}
 Adding these three expression together, we obtain the recursion
 \begin{equation}\label{eq:split2}
 \begin{split}
& \Ex_{k}\left[\frac{1}{2}\norm{\xx^{k+1}-\xx^{*}}^{2}_{\bm{P}\bm{Q}^{k}+\bm{P}\Upsilon}+\frac{S_{k}}{\sigma_{k}}\left(\hat{\Psi}_{k+1}-\Psi(\xx^{*})+S_{k}(h(\bm{w}^{k+1})-h(\xx^{*}))\right)\right]\\
&\leq \frac{1}{2}\norm{\xx^{k}-\xx^{*}}^{2}_{\bm{P}\bm{Q}^{k}+(\bm{P}-\bm{I})\Upsilon}+\frac{S_{k-1}}{\sigma_{k}}\left(\hat{\Psi}_{k}-\Psi(\xx^{*})+S_{k-1}(h(\bm{w}^{k})-h(\xx^{*}))\right)\\
&-\frac{1}{2}\norm{\hat{\xx}^{k+1}-\xx^{k}}^{2}_{\bm{Q}^{k}-(\Lambda+\sigma_{k}\Xi)}-\sigma_{k}(h(\xx^{k})-h(\xx^{*})).
 \end{split}
 \end{equation}
Recall that $\bm{Q}^{k}=\frac{1}{\tau_{k}}\bm{P}\bm{B}$. Condition \eqref{eq:MI1} guarantees that $\bm{Q}^{k}\succeq \Lambda+\sigma_{k}\Xi$. Multiplying both sides of \eqref{eq:split2} by $\sigma_{k}$, we obtain 
\begin{align*}
& \Ex_{k}\left[\frac{1}{2}\norm{\xx^{k+1}-\xx^{*}}^{2}_{\frac{\sigma_{k}}{\tau_{k}}\bm{P}^{2}\bm{B}+\sigma_{k}\bm{P}\Upsilon}+S_{k}\left(\hat{\Psi}_{k+1}-\Psi(\xx^{*})+S_{k}(h(\bm{w}^{k})-h(\bm{x}^{*}))\right)\right]\\
&\leq \frac{1}{2}\norm{\xx^{k}-\xx^{*}}^{2}_{\frac{\sigma_{k}}{\tau_{k}}\bm{P}^{2}\bm{B}+\sigma_{k}(\bm{P}-\bm{I})\Upsilon}+S_{k-1}\left(\hat{\Psi}_{k}-\Psi(\xx^{*})+S_{k-1}(h(\bm{w}^{k})-h(\xx^{*}))\right)\\
&-\frac{1}{2}\norm{\hat{\xx}^{k+1}-\xx^{k}}^{2}_{\frac{\sigma_{k}}{\tau_{k}}\bm{P}\bm{B}-\sigma_{k}(\Lambda+\sigma_{k}\Xi)}-\sigma^{2}_{k}(h(\xx^{k})-h(\xx^{*})).
 \end{align*}
Under condition \eqref{eq:MI2}, it follows that 
 $$
 \norm{\xx^{k+1}-\xx^{*}}^{2}_{\bm{W}_{k+1}}\leq\norm{\xx^{k+1}-\xx^{*}}^{2}_{\frac{\sigma_{k}}{\tau_{k}}\bm{P}^{2}\bm{B}+\sigma_{k}\bm{P}\Upsilon}.
 $$
Exploiting this relation in the penultimate display, using definitions \eqref{eq:Fk} and \eqref{eq:Vk}, one readily obtains 
\[
 \Ex_{k}\left[V_{k+1}(\xx^{*})\right]\leq V_{k}(\xx^{*})-\sigma^{2}_{k}(h(\xx^{k})-h(\xx^{*}))-\frac{1}{2}\norm{\hat{\xx}^{k+1}-\xx^{k}}^{2}_{\frac{\sigma_{k}}{\tau_{k}}(\bm{P}\bm{B}-\tau_{k}(\sigma_{k}\Xi+\Lambda))}
\]
\end{proof}
 
\subsection{The Convex case}
\label{sec:convex}
Suppose that $\Upsilon=0$, so that $\bm{W}_{k}=\frac{\sigma_{k}}{\tau_{k}}\bm{P}^{2}\bm{B}$. In this case, the matrix inequality \eqref{eq:MI2} is satisfied for any sequences $(\sigma_{k})_{k\geq 0},(\tau_{k})_{k\geq 0}$ satisfying $\frac{\sigma_{k}}{\tau_{k}}\geq\frac{\sigma_{k+1}}{\tau_{k+1}}$. In particular, this can be realised with the policy $\sigma_{k}=\sigma\tau_{k}$ for all $k\geq 0$, and some constant $\sigma\in(0,\min_{i\in[d]}\pi_{i}]$. This implies $\theta_{k}=\frac{\sigma}{1+k\sigma}$ for all $k\geq 0$. This specification satisfies all conditions needed for Lemma \ref{lem:average} to hold. We further set $\tau_{k}\equiv 1$. It only remains to see if matrix inequality \eqref{eq:MI1} is satisfied. In fact, this condition gives us a restriction on $\sigma$, reading as 
\begin{equation}\label{eq:MI_convex}
\bm{P}\bm{B}\succeq \sigma\Xi+\Lambda.
\end{equation}
We show later in this section that this matrix inequality is achievable. In this regime, we will show that we get an $O(1/k)$ iteration complexity estimate of the averaged sequence $(\bm{w}^{k})_{k}$, in terms of the objective function value. This extends the results reported in \cite{malitsky17,malitsky2019primaldual,luke18} to general random block-coordinate activation schemes. The following Assumptions shall be in place throughout this section:
 \begin{assumption}\label{ass:SaddleExists}
 Problem \eqref{eq:P} admits a saddle point, i.e. a pair $(\xx^{\ast},\bm{y}^{\ast})\in\R^{m}\times\R^{m}$ satisfying \eqref{eq:KKT}.
 \end{assumption}
 \begin{assumption}\label{ass:Solution}
The solution set of problem \eqref{eq:P} $\X^{\ast}$ is nonempty.
\end{assumption}

Thanks to eq. \eqref{eq:waverage}, we know that $\bm{w}^{k}\in\Conv(\xx^{0},\ldots,\xx^{k})\subseteq \domain(r)$. Therefore, 
\begin{align*}
R(\bm{w}^{k})-R(\xx^{\ast})&\geq \inner{-\nabla\phi(\xx^{*})-\bm{A}^{\top}\bm{A}\bm{y}^{*}}{\bm{w}^{k}-\xx^{*}}\\
&\geq \phi(\xx^{*})-\phi(\bm{w}^{k})+\inner{\bm{A}\bm{y}^{*}}{\bm{A}(\xx^{*}-\bm{w}^{k})}
\end{align*}
Whence, 
\begin{equation}\label{eq:LBw}
\Psi(\bm{w}^{k})-\Psi(\xx^{*})\geq-\norm{\bm{A}\bm{y}^{*}}\cdot\norm{\bm{A}(\bm{w}^{k}-\xx^{*})}=-\delta\sqrt{h(\bm{w}^{k})-h(\xx^{*})},
\end{equation}
where $\delta:=\sqrt{2}\norm{A\bm{y}^{*}}$. We assume that $\tau_{k}\equiv 1$ and $\sigma_{k}=\sigma>0$. The energy functions \eqref{eq:Fk} and \eqref{eq:Vk} take the form $F_{k}=\hat{\Psi}_{k}+(1+(k-1)\sigma)(h(\bm{w}^{k})-h(\xx^{*}))$,  and $V_{k}(\xx):=\frac{1}{2}\norm{\xx^{k}-\xx}^{2}_{\bm{P}^{2}\bm{B}}+(1+(k-1)\sigma)(F_{k}-\Psi(\xx))$. 
\begin{lemma}\label{lem:LBF}
Suppose that Assumption \ref{ass:SaddleExists} applies. Then, the process $\left(S_{k-1}(F_{k}-\Psi(\xx^{*}))\right)_{k\geq 0}$ is almost surely bounded.
\end{lemma}
\begin{proof}
Note that 
\begin{align*}
S_{k-1}(F_{k}-\Psi(\xx^{*}))&=S_{k-1}(\hat{\Psi}_{k}-\Psi(\xx^{*}))+S^{2}_{k-1}(h(\bm{w}^{k})-h(\xx^{*}))\\
&\geq S_{k-1}(\Psi(\bm{w}^{k})-\Psi(\xx^{*}))+S^{2}_{k-1}(h(\bm{w}^{k})-h(\xx^{*}))\\
&\overset{\eqref{eq:LBw}}{\geq} -\delta S_{k-1}\sqrt{h(\bm{w}^{k})-h(\xx^{*})}+S^{2}_{k-1}(h(\bm{w}^{k})-h(\xx^{*}))\\
&=\frac{S^{2}_{k-1}}{2}(h(\bm{w}^{k})-h(\xx^{*}))+\left(\frac{S^{2}_{k-1}}{2}(h(\bm{w}^{k})-h(\xx^{*}))-\delta S_{k-1}\sqrt{h(\bm{w}^{k})-h(\xx^{*})}\right)
\end{align*}
The convex function $t\mapsto \frac{S^{2}}{2}t^{2}-\delta S t$ attains the global minimum at the value $-\frac{\delta^{2}}{2}$. Hence, 
\begin{equation}\label{eq:convex1}
S_{k-1}(F_{k}-\Psi(\xx^{*}))\geq \frac{S^{2}_{k-1}}{2}(h(\bm{w}^{k})-h(\xx^{*}))-\frac{\delta^{2}}{2}\geq-\frac{\delta^{2}}{2}.
\end{equation}
The last inequality uses the fact that $\xx^{\ast}\in\X$, so that $h(\bm{w}^{k})\geq h(\xx^{\ast})$. This completes the proof.
\end{proof}
We are now in the position to give the proof of the first main result of this paper.
\begin{theorem}\label{thm:convex}
Suppose that Assumptions \ref{ass:SaddleExists} and \ref{ass:Solution} apply. Consider the sequence $(\bm{z}^{k},\bm{w}^{k},\xx^{k})_{k\geq 0}$ generated by Algorithm \ref{alg:RBCD} with $\sigma_{k}\equiv\sigma$ and $\tau_{k}\equiv 1$ for all $k\geq 1$. Then, the following holds: 
\begin{itemize}
\item[(i)] 
$h(\bm{w}^{k})-h(\xx^{*})=O(1/k^{2})$ a.s. 
\item[(ii)] $(\xx^{k})_{k}$ and $(\bm{w}^{k})_{k}$ converge a.s. to a random variable with values in $\X^{\ast}$.
\item[(iii)] $\abs{\Psi(\bm{w}^{k})-\Psi(\xx^{*})}=O(1/k)$ a.s. 
\end{itemize}
\end{theorem}
\begin{proof}
(i) From Lemma \ref{lem:LBF}, it follows that the process $\mathcal{E}_{k}:=V_{k}(\xx^{*})+\frac{\delta^{2}}{2}$ is non-negative. Eq. \eqref{eq:martingale} shows that $\mathcal{E}_{k}$ satisfies the recursion 
\begin{equation}\label{eq:SM}
\Ex_{k}[\mathcal{E}_{k+1}]\leq\mathcal{E}_{k}-\sigma^{2}(h(\xx^{k})-h(\xx^{*}))-\frac{1}{2}\norm{\hat{\xx}^{k+1}-\xx^{k}}^{2}_{\sigma(\bm{P}\bm{B}-(\sigma\Xi+\Lambda))}.
\end{equation}
Lemma \ref{lem:RS} implies that $(\mathcal{E}_{k})_{k\geq 0}$ converges a.s. to a limiting random variable $\mathcal{E}_{\infty}\in[0,\infty)$ and 
$$
\sum_{k=0}^{\infty} \left[\sigma^{2}(h(\xx^{k})-h(\xx^{*}))+\frac{1}{2}\norm{\hat{\xx}^{k+1}-\xx^{k}}^{2}_{\sigma(\bm{P}\bm{B}-(\sigma\Xi+\Lambda))}\right]<\infty\quad\text{a.s.}
$$
Since $h(\xx^{k})-h(\xx^{*})\geq 0$, it follows $\lim_{k\to\infty}\frac{1}{2}\norm{\hat{\xx}^{k+1}-\xx^{k}}^{2}_{\sigma(\bm{P}\bm{B}-(\sigma\Xi+\Lambda))}=0$. Consequently, the sequence $\left(\frac{1}{2}\norm{\hat{\xx}^{k+1}-\xx^{k}}^{2}_{\sigma(\bm{P}\bm{B}-(\sigma\Xi+\Lambda))}\right)_{k\geq 0}$ is bounded. 
By definition of the energy function $V_{k}(\xx^{*})$, the sequence $(\xx^{k})_{k\geq 0}$ is bounded and \eqref{eq:convex1} yields 
$h(\bm{w}^{k})-h(\xx^{*})\geq -\frac{\delta^{2}}{S^{2}_{k-1}}.$ Hence, $h(\bm{w}^{k})-h(\xx^{*})=O(1/k^{2})$ a.s. Moreover, it is easy to see that $\lim_{k\to\infty}\sum_{t=0}^{k}\gamma_{i}^{k,t}=1$ and $\lim_{k\to\infty}\gamma_{i}^{k,t}=0$ for all $i\in[d]$. Hence, by Lemma \ref{lem:average} and the Silverman-Toeplitz theorem, the sequence $(\bm{w}^{k})_{k}$ converges to the same limit point as $(\xx^{k})_{k}$, a.s.

By definition of the energy function $V_{k}$, the sequence $(\xx^{k})_{k\geq 0}$ is a.s. bounded. We show next that all accumulation bounds are contained in the set of solutions $\X^{\ast}$ of \eqref{eq:P}.

Let $\Omega_{1}$ be a set of probability 1 on which the subsequence $(\xx^{k_{j}}(\omega))_{j\in\Natural}$ converges to a limiting random variable $\bar{\xx}(\omega)$ for all $\omega\in\Omega_{1}$. It follows from $\lim_{j\to\infty}\norm{\hat{\xx}^{k_{j}}(\omega)-\xx^{k_{j}}(\omega)}=0$ on $\Omega_{1}$ that $(\hat{\xx}^{k_{j}})_{j\in\Natural}$ converges to $\bar{\xx}$ as well on $\Omega_{1}$. Furthermore, 
\begin{align*}
h(\bm{z}^{k})-h(\xx^{\ast})&=\theta_{k}[h(\bm{w}^{k})-h(\xx^{\ast})]+(1-\theta_{k})[h(\xx^{k})-h(\xx^{\ast})]\\
&-\frac{\theta_{k}(1-\theta_{k})}{2}\norm{\bm{A}(\xx^{k}-\bm{w}^{k})}^{2}\\
&=O(k^{-2}),
\end{align*}
and for all $\xx^{\ast}\in\X$, 
\begin{align*}
\inner{\nabla h(\bm{z}^{k})}{\hat{\xx}^{k+1}-\xx^{\ast}}&=\inner{\bm{A}^{\top}\bm{A}(\bm{z}^{k}-\xx^{\ast})}{\hat{\xx}^{k+1}-\xx^{\ast}}\\
&\leq\norm{\bm{A}(\bm{z}^{k}-\xx^{\ast})}\cdot\norm{\bm{A}(\hat{\xx}^{k+1}-\xx^{\ast})}\\
&\leq 2\sqrt{h(\bm{z}^{k})-h(\xx^{\ast})}\cdot\sqrt{h(\hat{\xx}^{k+1})-h(\xx^{\ast})}=o(k^{-1})
\end{align*}
Hence, $\lim_{k\to\infty}S_{k}\inner{\nabla h(\bm{z}^{k})}{\hat{\xx}^{k+1}-\xx^{\ast}}=0$ for all $\xx^{\ast}\in\X$ a.s. Using Lemma \ref{lem:xi} at the point $\xx^{\ast}\in\X^{\ast}\subset\X$ shows that 
\begin{align*}
&R(\hat{\xx}^{k_{j+1}})+\inner{\nabla\Phi(\xx^{k_{j}})}{\hat{\xx}^{k_{j+1}}-\xx^{k_{j}}}+S_{k}\inner{\nabla h(\bm{z}^{k_{j}})}{\hat{\xx}^{k_{j+1}}-\bm{z}^{k_{j}}}+\frac{1}{2}\norm{\hat{\xx}^{k_{j+1}}-\xx^{k_{j}}}^{2}_{\bm{Q}}\\
&\leq R(\xx^{\ast})+\inner{\nabla\Phi(\xx^{k_{j}})}{\xx^{\ast}-\xx^{k_{j}}}+S_{k}\inner{\nabla h(\bm{z}^{k_{j}})}{\xx^{\ast}-\bm{z}^{k_{j}}}+\frac{1}{2}\norm{\hat{\xx}^{\ast}-\xx^{k_{j}}}^{2}_{\bm{Q}}-\frac{1}{2}\norm{\xx^{\ast}-\hat{\xx}^{k_{j+1}}}_{\bm{Q}}.
\end{align*}
Let $j\to\infty$, and use the lower semi-continuity of the function $R(\cdot)$, one arrives that the inequality 
$$
R(\bar{\xx})\leq R(\xx^{\ast})+\inner{\nabla\Phi(\bar{\xx})}{\xx^{\ast}-\bar{\xx}}.
$$
Since $h(\xx^{k_{j}})-h(\xx^{\ast})\to 0$, the limit point $\bar{\xx}$ is an element of $\X$. Convexity of $\Phi$ in turn yields $\Phi(\xx^{\ast})\geq\Phi(\bar{\xx})+\inner{\nabla\Phi(\bar{\xx})}{\xx^{\ast}-\bar{\xx}}$. We conclude that $\Psi(\bar{\xx})\leq\Psi(\xx^{\ast})$ and $\bar{\xx}\in\X$. Therefore, $\bar{\xx}\in\X^{\ast}$.\\
\noindent
(ii) We next show that cluster points are unique and thereby demonstrate almost sure convergence of the last iterate. We argue by contradiction. Suppose there are converging subsequences $(\xx^{k})_{k\in\mathcal{K}_{1}}$ and $(\xx^{k})_{k\in\mathcal{K}_{2}}$ with limit points $\xx'$ and $\xx''$, respectively. Hence, $\Psi(\xx')=\Psi(\xx'')\equiv\Psi^{\ast}$. Following the above argument, we see that $\xx',\xx''\in\X^{\ast}$ a.s. The point $\xx^{\ast}$ chosen in the previous argument was arbitrary, so we can replace it by $\xx'$. Let us simplify notation by setting $a_{k}:=S_{k}(F_{k}-\Psi^{\ast})$. Since $V_{k}(\xx')$ converges almost surely, we see 
\begin{align*}
\lim_{k\to\infty}V_{k}(\xx')&=\lim_{k\to\infty,k\in\mathcal{K}_{1}}V_{k}(\xx')=\lim_{k\to\infty,k\in\mathcal{K}_{1}}\left(\frac{1}{2}\norm{\xx^{k}-\xx'}^{2}_{\bm{W}}+a_{k}\right)\\
&=\lim_{k\to\infty,k\in\mathcal{K}_{1}}a_{k}.
\end{align*}
Similarly, 
\begin{align*}
\lim_{k\to\infty}V_{k}(\xx')&=\lim_{k\to\infty,k\in\mathcal{K}_{2}}V_{k}(\xx')=\lim_{k\to\infty,k\in\mathcal{K}_{2}}\left(\frac{1}{2}\norm{\xx^{k}-\xx'}^{2}_{\bm{W}}+a_{k}\right)\\
&=\norm{\xx''-\xx'}_{\bm{W}}+\lim_{k\to\infty,k\in\mathcal{K}_{2}}a_{k}.
\end{align*}
We conclude that 
$$
\lim_{k\to\infty,k\in\mathcal{K}_{1}}a_{k}=\norm{\xx''-\xx'}_{\bm{W}}+\lim_{k\to\infty,k\in\mathcal{K}_{2}}a_{k}.
$$
Repeating the same analysis for $\xx''$ instead of $\xx'$, we get 
$$
\lim_{k\to\infty,k\in\mathcal{K}_{2}}a_{k}=\norm{\xx''-\xx'}_{\bm{W}}+\lim_{k\to\infty,k\in\mathcal{K}_{1}}a_{k}.
$$
Combining these two equalities, we see that $\xx'=\xx''$. Therefore, the whole sequence $(\xx^{k})_{k}$ converges pointwise almost surely to a solution. Lemma \ref{lem:average} yields the same assertion for $(\bm{w}^{k})_{k}$.\\
\noindent 
(iii) Taking expectations on both sides of \eqref{eq:martingale} and iterating the thus obtained expression gives 
$$
\Ex\left[\frac{1}{2}\norm{\xx^{k}-\xx^{*}}^{2}_{\bm{W}}+S_{k-1}(\hat{\Psi}_{k}-\Psi(\xx^{*}))+S^{2}_{k-1}(h(\bm{w}^{k})-h(\xx^{*}))\right]\leq \frac{1}{2}\norm{\xx^{0}-\xx^{*}}^{2}_{\bm{W}}.
$$
From here, we deduce 
$$
S_{k-1}\Ex[\hat{\Psi}_{k}-\Psi(\xx^{*})]\leq \frac{1}{2}\norm{\xx^{0}-\xx^{*}}^{2}_{\bm{W}}
$$
which gives $\hat{\Psi}_{k}-\Psi(\xx^{*})\leq O(1/k)$ a.s. Since $\hat{\Psi}_{k}\geq \Psi(\bm{w}^{k})$, it follows $\Psi(\bm{w}^{k})-\Psi(\xx^{*})\leq O(1/k)$ a.s.
To obtain a lower bound, it suffices to apply eq. \eqref{eq:LBw} and use the results from (i) to obtain 
$$
\Psi(\bm{w}^{k})-\Psi(\xx^{\ast})\geq -\delta\sqrt{h(\bm{w}^{k})-h(\xx^{\ast})}=O(1/k)\quad \text{a.s.}
$$
\end{proof}
We conclude this section with a concrete example showing that the matrix inequality \eqref{eq:MI_convex} can be satisfied.
\begin{example}
Assume that $\bm{\Xi}=\blkdiag(\pi_{1}^{-1}\bm{A}_{1}^{\top}\bm{A}_{1},\ldots,\pi_{d}^{-1}\bm{A}_{d}^{\top}\bm{A}_{d})$. In this case eq. \eqref{eq:MI_convex} can be decomposed to the block specific condition $\pi^{-1}_{i}\bm{B}_{i}\succeq \Lambda_{i}+\frac{\sigma}{\pi_{i}}\bm{A}_{i}^{\top}\bm{A}_{i}$. Let us assume further that $\Lambda_{i}=L_{i}\bm{I}_{m_{i}}$, for a scalar $L_{i}>0$. If we choose $\bm{B}_{i}=\pi_{i}(\lambda_{i}+L_{i})\bm{I}_{m_{i}}$, where $\lambda_{i}=\norm{\bm{A}_{i}}_{2}$. Then, a sufficient condition for satisfying the matrix inequality is $\lambda_{i}+L_{i}\geq L_{i}+\frac{\sigma}{\pi_{i}}\lambda_{i}$ for all $i\in[d]$. This inequality can be satisfied by choosing $\sigma=\min_{i\in[d]} \pi_{i}$, which is the largest value for a given set of activation probabilities, that is compatible with Lemma \ref{lem:average}. 
\end{example}
 \subsection{The strongly convex case}
 \label{sec:strongconvex}
 In this section we study the performance of Algorithm \ref{alg:RBCD} in the strongly convex regime.
 \begin{assumption}\label{ass:strongconvex}
 $\Upsilon\succ 0$.
 \end{assumption}
The main challenge we need to overcome is to design step size sequences which obey the matrix inequalities \eqref{eq:MI1} and \eqref{eq:MI2}. Let us focus on \eqref{eq:MI1} first and show how it can be restated in a more symmetric way. To that end, define the matrix $\bm{B}=\bm{P}^{-2}\Upsilon$ and $\bm{M}^{1/2}:=\bm{P}^{-1/2}\Upsilon^{1/2}$, so that $\bm{M}=\bm{P}^{-1}\Upsilon$. In terms of this new parameter matrix, we see 
\begin{equation}\label{eq:MI1-1}
\bm{P}^{-1}\Upsilon\succeq \tau_{k}\Lambda+\sigma_{k}\tau_{k}\bm{\Xi}\iff \frac{1}{\tau_{k}}\bm{I}\succeq \bm{M}^{-1/2}\Lambda\bm{M}^{-1/2}+\sigma_{k}\bm{M}^{-1/2}\bm{\Xi}\bm{M}^{-1/2}.
\end{equation}
It is easy to see that $\bm{\Xi}$ is symmetric. Furthermore, for $\bm{u}\in\R^{m}\setminus\{0\}$, we see $\bm{u}^{\top}\bm{\Xi}\bm{u}=\Ex[\norm{\bm{A}\bm{t}}^{2}]\geq 0$, where $\bm{t}:=\bm{P}\bm{E}_{k}\bm{u}$ is a random vector in $\R^{m}$ with mean $\bm{u}$. It follows $\bm{\Xi}\in\symm^{m}_{+}$. This suggests to relate the step size parameters to the spectra of the involved matrices. The rest of this section assumes the following parametric model on the coupling between $\sigma_{k}$ and $\tau_{k}$.
\begin{assumption}\label{ass:coupling}
The sequence $(\tau_{k})_{k}$ is non-increasing and positive. They are related by the coupling equation 
\begin{equation}\label{eq:coupling}
\sigma_{k}\tau_{k}=\alpha-\beta\tau_{k}\qquad\forall k\geq 0,
\end{equation}
where $\alpha$ and $\beta$ are positive numbers. 
\end{assumption}
We note that this coupling relation automatically means that $\sigma_{k}=\alpha/\tau_{k}-\beta$ is non-decreasing. A sufficient condition to satisfy eq. \eqref{eq:MI1-1} is 
 $$
 \frac{1}{\tau_{k}}\geq \lambda_{\max}(\bm{M}^{-1/2}\Lambda\bm{M}^{-1/2})+\sigma_{k}\lambda_{\max}(\bm{M}^{-1/2}\Xi\bm{M}^{-1/2}).
 $$
 Using the model \eqref{eq:coupling}, this boils down to 
 $$
 \frac{1}{\tau_{k}}(1-\alpha\lambda_{\max}(\bm{M}^{-1/2}\bm{\Xi}\bm{M}^{-1/2})) \geq \lambda_{\max}(\bm{M}^{-1/2}\Lambda\bm{M}^{-1/2})-\beta\lambda_{\max}(\bm{M}^{-1/2}\bm{\Xi}\bm{M}^{-1/2}).
 $$
Choosing 
\begin{equation}\label{eq:alpha}
\alpha:=\lambda_{\max}(\bm{M}^{-1/2}\bm{\Xi}\bm{M}^{-1/2})^{-1},
\end{equation}
 we get 
$$
\beta\geq \frac{\lambda_{\max}(\bm{M}^{-1/2}\Lambda\bm{M}^{-1/2})}{\lambda_{\max}(\bm{M}^{-1/2}\bm{\Xi}\bm{M}^{-1/2})}=\alpha\lambda_{\max}(\bm{M}^{-1/2}\Lambda\bm{M}^{-1/2}).
$$
We make the choice 
\begin{equation}\label{eq:beta}
\beta=\alpha\lambda_{\max}(\bm{M}^{-1/2}(\Lambda+\Upsilon)\bm{M}^{-1/2}),
\end{equation}
so that the coupling equation \eqref{eq:coupling} gives rise to a step size process satisfying matrix inequality \eqref{eq:MI1}. 
\begin{remark}
To get some intuition behind the derived conditions for $\alpha$ and $\beta$, consider the independent sampling case under which the operator $\bm{\Xi}$ factorises to $\bm{\Xi}=\blkdiag(\frac{1}{\pi_{1}}\bm{A}_{1}^{\top}\bm{A}_{1},\ldots,\frac{1}{\pi_{d}}\bm{A}_{d}^{\top}\bm{A}_{d})$. Further, suppose $\Lambda=\blkdiag[L_{1}\bm{I}_{m_{1}};\ldots;L_{d}\bm{I}_{m_{d}}]$. Then, the above choice says that 
\begin{align}
&\alpha=\left(\max_{i\in[d]}\frac{1}{\mu_{i}\pi_{i}^{2}}\norm{\bm{A}_{i}}_{2}^{2}\right)^{-1}, \text{and }\\
&\kappa:=\frac{\beta}{\alpha}=\max_{i\in[d]}\frac{L_{i}+\mu_{i}}{\pi_{i}\mu_{i}}.
\end{align}
Clearly, $\kappa$ is related to the condition number of problem \eqref{eq:P}, and satisfies $\kappa\geq\frac{1}{\pi_{i}}$ for all $i\in[d]$. We point out that the explicit expression for $\kappa$ does not rely on the independent sampling assumption. Only the quantification of the parameter $\alpha$ exploits this special structure.
\end{remark}
Let us fix these identified values for $\alpha$ and $\beta$, and continue our derivation with matrix inequality \eqref{eq:MI2}. Let us choose $\bm{B}=\bm{P}^{-2}\Upsilon=\blkdiag[\pi_{1}^{2}\mu_{1}\bm{I}_{m_{1}};\ldots;\pi_{d}^{2}\mu_{d}\bm{I}_{m_{d}}]$. Using the block structure, and the specification of the sequence $(\sigma_{k})_{k}$, it can be verified that \eqref{eq:MI2} reduces to the scalar-valued inequality
\begin{equation}\label{eq:tau}
\tau^{2}_{k+1}\left[(\alpha-\beta\tau_{k})(\pi_{i}+\tau_{k})+\beta(1-\pi_{i})\tau^{2}_{k}\right]-\tau_{k+1}\tau^{2}_{k}[\alpha(1-\pi_{i})-\beta\pi_{i}]-\tau^{2}_{k}\alpha\pi_{i}\geq 0.
\end{equation}
This defines a quadratic inequality in $\tau_{k+1}$ of the form $c_{i,1}(\tau_{k})\tau^{2}_{k+1}+c_{i,2}(\tau_{k})\tau_{k+1}-c_{i,3}(\tau_{k})\geq 0$, with coefficients 
\begin{align*}
&c_{i,1}(t):=(\alpha-\beta t)(\pi_{i}+t)+\beta(1-\pi_{i})t^{2},\\
&c_{i,2}(t):=t^{2}[\beta\pi_{i}-\alpha(1-\pi_{i})],\\
&c_{i,3}(t):=t^{2}\alpha\pi_{i}
\end{align*}
This suggests to use the recursive step-size policy
\begin{equation}\label{eq:tau2}
\tau_{k+1}=\max_{i\in[d]} \frac{\frac{\tau^{2}_{k}}{2}(1/\pi_{i}-\kappa-1)+\tau_{k}\sqrt{\left(1+(1/\pi_{i}-\kappa)\frac{\tau_{k}}{2}\right)^{2}-\left(\frac{2-\pi_{i}}{\pi_{i}}+2\kappa\right)\frac{\tau^{4}_{k}}{4}}}{1+\tau_{k}(1/\pi_{i}-\kappa)-\kappa\tau^{2}_{k}}.
\end{equation}
We study the qualitative properties of the so produced sequence $(\tau_{k})_{k}$ in detail in Appendix \ref{app:parameters}, under the following assumptions:
\begin{assumption}\label{ass:Lambda}
For all $i\in[d]$ we have $\Lambda_{i}=L_{i}\bm{I}_{m_{i}}$. 
\end{assumption}
\begin{assumption}\label{ass:Usample}
The sampling is uniform: $\pi_{i}\equiv \pi\in(0,1)$ for all $i\in[d]$. 
\end{assumption}
Uniform sampling is a very common sampling scheme employed in the literature. It contains as special cases the single coordinate activation scheme (i.e. $\pi_{i}=\frac{1}{d}$), and the uniform sampling scheme. In particular, the frequently employed $m$-nice sampling (cf. Section \ref{sec:sampling}) is contained in this framework. Under this assumption we prove that if $\tau_{0}$ is chosen sufficiently small, then the sequence $(\tau_{k})_{k}$ exhibits the same qualitative behaviour as a Nesterov accelerated method \cite{Nes18,dAspremont:2021aa}. This is summarised in Lemma \ref{lem:stepsizes}, whose proof is provided in Appendix \ref{app:parameters}.
\revise{
\begin{lemma}
\label{lem:stepsizes} 
Let Assumptions \ref{ass:coupling}, \ref{ass:Lambda} and \ref{ass:Usample} hold, and consider the step size sequences $(\sigma_{k})_{k},(\tau_{k})_{k}$ constructed recursively via eq. \eqref{eq:tau2}, with initial condition $\tau_{0}\in(0,1/\kappa)$. Then, $(\tau_{k})_{k}$ is monotonically decreasing and satisfies 
$$
\tau_{k}\geq \frac{2\tau_{0}}{(1+\kappa-1/\pi)\tau_{0}k+2}\qquad\forall k\geq 0.
$$
In particular, $\sigma_{k}=O(k)$ and $S_{k}=O(k^{2})$.
\end{lemma}
}
We therefore can prove accelerated rates for our scheme in the strongly convex case.\\
\begin{theorem}\label{thm:sconvex}
Suppose that Assumptions \ref{ass:SaddleExists}-\ref{ass:Usample} apply. Consider the sequence $(\bm{z}^{k},\bm{w}^{k},\xx^{k})_{k\geq 0}$ generated by Algorithm \ref{alg:RBCD} with the following step size policy constructed via eqs. 
\eqref{eq:coupling}, \eqref{eq:alpha}, \eqref{eq:beta}. Then, the following holds: 
\begin{itemize}
\item[(i)] We have $h(\xx^{k})-h(\xx^{*})=o(k^{-2})$ and $h(\bm{w}^{k})-h(\xx^{*})=O(k^{-4})$ a.s. 
\item[(ii)] $(\xx^{k})_{k}$ and $(\bm{w}^{k})_{k}$ converge a.s. to the solution $\mathcal{X}^{\ast}=\{\bar{\xx}\}$.
\item[(iii)] $\abs{\Psi(\bm{w}^{k})-\Psi(\xx^{*})}=O(k^{-2})$ a.s. 
\end{itemize}
\end{theorem}
\begin{proof}
\noindent 
(i) If a Lagrange multiplier exists then Lemma \ref{lem:LBF} still applies and guarantees that the function $V_{k}(\xx^{\ast})+\frac{\delta^{2}}{2}\geq 0$ for all $k\geq 0$, and the recursion \eqref{eq:SM} applies, which reads in the present case 
$$
\Ex_{k}[\mathcal{E}_{k+1}]\leq\mathcal{E}_{k}-\sigma^{2}_{k}(h(\xx^{k})-h(\xx^{*}))-\frac{\sigma_{k}}{2\tau_{k}}\norm{\hat{\xx}^{k+1}-\xx^{k}}^{2}_{\bm{P}^{-1}\Upsilon-\tau_{k}(\sigma_{k}\Xi+\Lambda)}.
$$
Since $\bm{P}^{-1}\Upsilon-\tau_{k}(\sigma_{k}\Xi+\Lambda)\succeq 0$, we can upper bound the right-hand side of the above display as 
\begin{equation}\label{eq:SM2}
\Ex_{k}[\mathcal{E}_{k+1}]\leq\mathcal{E}_{k}-\sigma^{2}_{k}(h(\xx^{k})-h(\xx^{*}))-\frac{\sigma_{k}}{2\tau_{k}}\norm{\hat{\xx}^{k+1}-\xx^{k}}^{2}_{\bm{P}^{-1}\Upsilon}.
\end{equation}
The supermartingale convergence theorem \revise{(Lemma \ref{lem:RS})} implies that 
$$
\lim_{k\to\infty}\sigma^{2}_{k}(h(\xx^{k})-h(\xx^{*}))=0 \text{ and }\lim_{k\to\infty}\frac{\sigma_{k}}{2\tau_{k}}\norm{\hat{\xx}^{k+1}-\xx^{k}}^{2}_{\bm{P}^{-1}\Upsilon}=0.
$$
Since $\sigma_{k}=O(k)$, it follows $h(\xx^{k})-h(\xx^{\ast})=o(k^{-2})$. Furthermore, the sequence $(\xx^{k})_{k}$ is bounded and eq. \eqref{eq:convex1} yields $h(\bm{w}^{k})-h(\xx^{\ast})= O(k^{-4})$, since $S_{k}=O(k^{2})$ as shown in Appendix \ref{app:parameters}. Moreover, $(\bm{w}^{k})_{k}$ and $(\xx^{k})_{k}$ share the same limit. \\
\noindent 
(ii) We can repeat the arguments of Theorem \ref{thm:convex}, to conclude that accumulation points of $(\hat{\xx}^{k})_{k}$ and $(\xx^{k})_{k}$ converge to the same random variable with values in $\mathcal{X}^{\ast}$. However, since the problem is strongly convex, we must have $\mathcal{X}^{\ast}=\{\bar{\xx}\}$ for some $\bar{\xx}\in\mathcal{X}$. Hence, the sequences actually converge and the common limit point is $\bar{\xx}$ a.s. 
\noindent 
(iii) We next show convergence rates in terms of the objective function gap. Iterating eq. \eqref{eq:SM2}, it follows $\Ex[V_{k}(\xx^{\ast})]\leq \frac{1}{2}\norm{\xx^{0}-\xx^{\ast}}^{2}_{\bm{W}_{0}}$. This is equivalent to 
$$
\Ex\left[\frac{1}{2}\norm{\xx^{k}-\xx^{\ast}}^{2}_{\bm{W}_{k}}+S_{k-1}(\hat{\Psi}_{k}-\Psi(\xx^{\ast}))+S^{2}_{k-1}(h(\bm{w}^{k})-h(\xx^{\ast}))\right]\leq \frac{1}{2}\norm{\xx^{0}-\xx^{\ast}}^{2}_{\bm{W}_{0}}.
$$
In particular, 
$$
\Ex[\hat{\Psi}_{k}-\Psi(\xx^{\ast})]\leq \frac{1}{2S_{k-1}}\norm{\xx^{0}-\xx^{\ast}}^{2}_{\bm{W}_{0}}=O(k^{-2}).
$$
Furthermore, $\hat{\Psi}_{k}\geq\Psi(\bm{w}^{k})$, so that via eq. \eqref{eq:LBw} we obtain
\begin{align*}
\Ex[\Psi(\bm{w}^{k})-\Psi(\xx^{\ast})]\leq O(k^{-2})\text{ and }\Ex[\Psi(\bm{w}^{k})-\Psi(\xx^{\ast})]\geq -\delta\Ex\left[\sqrt{h(\bm{w}^{k})-h(\xx^{\ast})}\right].
\end{align*}
Combining these bounds, it follows $\Ex[\Psi(\bm{w}^{k})-\Psi(\xx^{\ast})]=O(k^{-2})$.
\end{proof}

\section{Application to power systems}
\label{sec:application}
In this section we apply Algorithm \ref{alg:RBCD} to the distributed coordination of an energy grid in an AC-optimal power flow formulation. Specifically, we illustrate how our block-coordinate descent method provides an efficient way to replicate the transmission-level locational marginal prices to the distribution level.\footnote{A preliminary version of this application is studied in our conference paper \cite{DLMP-CDC}. We refer to this paper for background and further motivation.} 

\subsection{Power flow model}
Consider a radial network with $\nodenb$ buses $\Node=\{1,\dots,\nodenb\}$, excluding the slack node $0$. Hence, the distribution network assumes the structure of a tree graph and we identify the transmission lines with the label of the parental node. The network is optimised over a time window $\Tim=\{1,\dots,\timlast\}$. We use $\napcn{\node}=(\napcnt{\node}{1},\dots,\napcnt{\node}{\timlast})$ and $\nrpcn{\node}=(\nrpcnt{\node}{1},\dots,\nrpcnt{\node}{\timlast})$ to denote active and reactive power net consumption at bus $n$ at each time point $t\in\Tim$. We let $\imaginary=\sqrt{-1}$ and write a complex number $z\in\mathbb{C}$ as $\text{Re}(z)+\imaginary\text{Im}(z)$. Thus, $\napcnt{\node}{\tim}<0$ means that there is production of energy at bus $n$ at time $t$. At the slack node $\node=0$, we assume that power will only be generated and there is no consumption, i.e. $\napcnt{0}{\tim}\leq 0$. At node $n$, we denote by $f_{n}=(f_{n,t})_{t\in\Tim}$ and $g_{n}=(g_{n,t})_{t\in\Tim}$ the real and reactive power flows, so that $f_{n}+\imaginary g_{n}$ is the complex power injected into node $n$. By Ohm's law, we have $g_{n}=\text{Im}(V_{n}\bar{I}_{n})$ and $f_{n}=\text{Re}(V_{n}\bar{I}_{n})$. We let $\ell_{n}=(\ell_{n,t})_{t\in\Tim}=\abs{V_{n}}$ is the squared current magnitude, $R_{n}$ and $X_{n}$ denotes the series resistance and reactance. Hence, $z_{n}=R_{n}+\imaginary X_{n}$ is the series impedance, and $y_{n}=z_{n}^{-1}$ the shunt admittance. 

The load flow equations using the BFM \cite{BarWu89c,Mol19}, without explicit consideration of transmission tap ratios, are as follows:

\begin{subequations}\label{branchflowconstraints}
\newcommand{\spl}{}%
\begin{align}
 \label{flowconservationactive}
\spl & \apfn{\node} \,{-}\, \!\!\sum\limits_{\altnode:\altnodem=\node}\!\! ( \apfn{\altnode} \,{-}\, \resn{\altnode}\curn{\altnode} ) \,{+}\, \napcn{\node} \,{+}\, \conn{\node}\voln{\node} \,{=}\, 0
\hspace{-20mm} &[\dlmpn{\node}]
  \\
  \label{flowconservationreactive} 
\spl & \rpfn{\node} \,{-}\, \!\!\sum\limits_{\altnode:\altnodem=\node}\!\! ( \rpfn{\altnode} \,{-}\, \resn{\altnode}\curn{\altnode} ) \,{+}\, \nrpcn{\node} \,{-}\, \susn{\node}\voln{\node} \,{=}\, 0 \hspace{-20mm}& [\dlmqn{\node}]
  \\
   \label{Ohm} 
\spl &\voln{\node} \,{-}\, 2 (\resn{\node}\apfn{\node} \,{+}\, \rean{\node}\rpfn{\node}) \,{+}\, (\resn{\node}^2 \,{+}\, \rean{\node}^2) \,\curn{\node} \,{=}\, \voln{\nodem} 
\hspace{-4mm}& 
  \\
 \label{powerflowrelaxed}
\spl &  \apfnt{\node}{\tim}^2 +\rpfnt{\node}{\tim}^2 \leq \volnt{\node}{\tim} \curnt{\node}{\tim}
   &\forall \tim\in\Tim
  \\
  \label{powerflowbound}
\spl &  \apfnt{\node}{\tim}^2 +\rpfnt{\node}{\tim}^2 \leq \ubpfn{\node}^2
  & \forall \tim\in\Tim 
  \\
  \label{powerflowboundm}
\spl & (\apfnt{\node}{\tim}-\resn{\node}\curnt{\node}{\tim} )^2 +(\rpfnt{\node}{\tim}-\rean{\node}\curnt{\node}{\tim} )^2 \leq \ubpfn{\node}^2 \hspace{-20mm}&  \forall\tim\in\Tim
  \\
 \label{voltagebound}
\spl & \lbvoln{\node}\leq \volnt{\node}{\tim} \leq \ubvoln{\node} ,
  & \forall  \tim\in\Tim
 \end{align}
\end{subequations}  
where
\begin{itemize}
\item 
$\voln{\node}=(\volnt{\node}{1},\dots,\volnt{\node}{\timlast})$ and~$\voln{\nodem}$ are the squared voltage magnitudes at buses~$\node$ and~$\nodem$,
\item
$\curn{\node}$ is the squared current magnitude on branch $(\node,\nodem)$,
\item
$\apfn{\node}$ and $\rpfn{\node}$ are the active and the reactive parts of the power flow over line $(\node,\nodem)$,
\item
$\resn{\node}$ and~$\rean{\node}$ are the resistance and the reactance of branch $(\node,\nodem)$,
\item
$\conn{\node}$ and~$\susn{\node}$ are the line conductance and susceptance at~$\node$.
\end{itemize}
Equation \eqref{flowconservationactive} and \eqref{flowconservationreactive} are the active and reactive flow conservation equations, ~\eqref{Ohm} is an expression of Ohm's law for the branch $(\node,\nodem)$, and~\eqref{powerflowrelaxed} is a SOCP relaxation of the definition of the power flow, which implies $\apfnt{\node}{\tim}^2 +\rpfnt{\node}{\tim}^2 = \volnt{\node}{\tim} \curnt{\node}{\tim}$. There exist sufficient conditions under which the optimization problem subject to~\eqref{branchflowconstraints} remains exact, see \cite{farivar13,BienstockCascades,gan15}.
Equations~\eqref{powerflowbound} and  \eqref{powerflowboundm} are limitations on the  squared power flow magnitude on $(\node,\nodem)$, and~\eqref{voltagebound} gives lower and upper bounds on the voltage at~$\node$. For the coupling flow conservation laws, dual variables are attached, which are the DLMPs corresponding to active and reactive power.

For later reference, we point out that the network flow constraints \eqref{flowconservationactive}-\eqref{flowconservationreactive} can be compactly summarised
ed as $\bm{A}_{0}\xx_{0}+\sum_{a}\bm{A}_{a}\xx_{a}=\bm{b}$ for suitably defined matrices $\bm{A}_{0},\bm{A}_{a}$ and right-hand side vector $\bm{b}$. However, for our computational scheme to work, we do not need to assume that the linear constraint holds exactly. Instead, we assume that the linear relation $\bm{A}_{0}\xx_{0}+\sum_{a}\bm{A}_{a}\xx_{a}=\bm{b}$ holds inexactly. Specifically, motivated by data-driven approaches for power flow models \cite{Mezghani:2020aa}, we assume that there exists a random variable $\xi$ with bounded support such that 
$$
\bm{A}_{0}\xx_{0}+\sum_{a}\bm{A}_{a}\xx_{a}=\bm{b}+\xi,\text{ and }\frac{1}{2}\norm{\bm{A}_{0}\xx_{0}+\sum_{a}\bm{A}_{a}\xx_{a}-\bm{b}}^{2}\leq\delta.
$$
In such inexact regimes, conventional deterministic OPF solvers are not applicable, whereas Algorithm \ref{alg:PPDLMP} is designed to handle such scenarios.

\subsection{Load aggregators\acused{LA}}
The set of buses~$\Node$ is partitioned into a collection~$(\Nodea{\aggregator})_{\aggregator\in\Aggregator}$ of subsets, such that each node subset~$\Nodea{\aggregator}$ is managed by a \acl{LA}~$\aggregator\in\Aggregator$. Each \ac{LA} controls the flexible net power consumption ($p_{n,t}$) and generation at each node $\node\in\Nodea{\aggregator}$, given at time~$\tim$ by
\begin{subequations}\label{aggregatorconstraints}
\begin{equation}\label{netactivepowerconsumed}
 \napcnt{\node}{\tim} = \apcnt{\node}{\tim} -\appnt{\node}{\tim}
 ,\qquad 
 \nrpcnt{\node}{\tim} = \rpcnt{\node}{\tim} -\rppnt{\node}{\tim}
 ,
\end{equation}
\noeqref{netactivepowerconsumed}%
for all $\node\in\Node$ and $\tim\in\Tim.$ $\apcnt{\node}{\tim}\geq 0$ is the consumption part and $\appnt{\node}{\tim}\geq 0$ is the production part of the power profile. Power consumption and production at the nodes are made flexible by the presence of deferrable loads (electric vehicles, water heaters) and \acdefp{DER}. The consumption at each node $\node\in\Node$ must satisfy a global energy demand $\lbmapcn{\node}$ over the full time window,
\begin{equation} \label{totalactivepowerconsumed}
 \sum_{\tim\in\Tim}\apcnt{\node}{\tim}\geq \lbmapcn{\node}
 , \qquad \forall \node\in\Node.
\end{equation}
\noeqref{totalactivepowerconsumed}%
Consumption and production are also constrained by power bounds and active to reactive power ratio:
\begin{align}\label{boundactivepowerconsumed}
& \lbapcnt{\node}{\tim} \leq \apcnt{\node}{\tim}\leq \ubapcnt{\node}{\tim}
 , &\forall \node\in\Node,\ \forall \tim\in\Tim,
\\
\label{reactivepowerconsumed} 
&\rpcnt{\node}{\tim} = \ratiopcn{\node}\apcnt{\node}{\tim}
 , & \forall \node\in\Node,\ \forall \tim\in\Tim,
 \\
\label{activepowerproduced}
&0 \leq \appnt{\node}{\tim} \leq \ubapc
 , & \forall \node\in\Node,\ \forall \tim\in\Tim,
\\
\label{reactivepowerproduced}
&\lbratioappnt{\node}{\tim}\appnt{\node}{\tim} \leq \rppnt{\node}{\tim} \leq \ubratioappnt{\node}{\tim}\appnt{\node}{\tim}
 , & \forall \node\in\Node,\ \forall \tim\in\Tim.
\end{align}
\noeqref{boundactivepowerconsumed}%
\noeqref{reactivepowerconsumed}%
\noeqref{activepowerproduced}%
\noeqref{reactivepowerproduced}%
\end{subequations}
Constraints \eqref{netactivepowerconsumed}-\eqref{reactivepowerproduced} define the feasible set $\XLA{\aggregator}$ of \ac{LA} decisions, containing vectors $\xLA{\aggregator}=(\napcn{\node},\nrpcn{\node})_{\node\in\Nodea{\aggregator}}$. Both, consumption and production, must be scheduled by the \ac{LA}, taking into account the current spot market prices, and other specific local factors characterising the private objectives of the \ac{LA}. Formally, there is a convex cost function $\costLA{\aggregator}{\xLA{\aggregator}}$ which the \ac{LA} would like to unilaterally minimise, subject to private feasibility $\xLA{\aggregator}\in\XLA{\aggregator}$.

\subsection{The distribution system operator}
In order to guarantee stability of the distribution network, the \ac{DSO} takes the individual aggregators' decisions into account and adjusts the power flows so that the flow conservation constraints \eqref{flowconservationactive}-\eqref{flowconservationreactive}, together with the SOCP constraints \eqref{Ohm}-\eqref{voltagebound}, are satisfied. Let $\xDSO=(\napcn{0},\nrpcn{0},\vapf,\vrpf,\vvol,\vcur)$ denote the vector of the variables controlled by the \ac{DSO}, and define the \ac{DSO}'s feasible set 
$\XDSO = \{ \xDSO \vert \eqref{Ohm}-\eqref{voltagebound}\text{ hold for }\node\in\Node \}. $ Then, the set of \ac{DSO} decision variables inducing a physically meaningful network flow for a given tuple of \ac{LA} decisions $\xLAs$ is described as 
$
\mathcal{F}(\xLAs)=\{\xDSO\in\XDSO\vert\eqref{flowconservationactive}-\eqref{flowconservationreactive}\text{ hold for }\xLAs\}.
$
Denoting the \ac{DSO} cost function $\phi_{0}(\xDSO)$, we arrive at the \ac{DSO}'s decision problem 
$$
\Psi(\xLAs)=\min\{\phi_{0}(\xDSO)\vert \xDSO\in\mathcal{F}(\xLAs)\},
$$
This represents the smallest costs to the \ac{DSO}, given the profile of flexible net consumption and generation at each affiliated node $n\in\Node_{\aggregator}$.

\subsection{A \acl{PPDLMP} }
The \ac{PPDLMP}, described in Algorithm \ref{alg:PPDLMP}, asks the \ac{DSO} to adjust \acsp{DLMP} based on the prevailing plans reported by the \acp{LA}. Once the price update is completed, a single \ac{LA} is selected at random to adapt the power profile within the subnetwork they manage. The local update of the \ac{LA} results in bid vector $w^{k}$, which will be fed into the \ac{DSO} final computational step to perform dispatch. 

\begin{algorithm}[t]
\caption{\acl{PPDLMP} (PPDLMP)}
\label{alg:PPDLMP}
\SetAlgoNoLine%
\SetKwInOut{Parameters}{\texttt{\textbf{Parameters}}}
\SetKw{Output}{\texttt{\textbf{Output:}}}
\SetKwFor{Initatdo}{\texttt{Initialization at}}{\texttt{:}}{}
\SetKw{phantomInit}{\phantom{\textbf{Initialization:}}}
\SetKwFor{For}{\texttt{for}}{\texttt{do}}{}
\SetKwFor{Atdo}{\texttt{at}}{\texttt{do}}{}
\Parameters{$p=\abs{\mathcal{A}}$, $\sigma>0$, $\bm{B}_{0}$, $\bm{B}_{a}$
}
\Initatdo{\texttt{\textup{\textbf{each aggregator $a\in\mathcal{A}$}}}}{
$\xx_{0}^{0}\in\mathcal{X}_{0}$\;
\texttt{\textbf{send bid}}~$u_{a}=\bm{A}_{a}\xx^{0}_{a}-b_{a}$ \texttt{\textbf{to the  \ac{DSO}}}
}
\Initatdo{\texttt{\textbf{the \ac{DSO}}}}{ 
$\xx^{0}_{0}\in\mathcal{X}_{0},\;\bm{v}^{0}= \sigma\sum_{a\in\mathcal{A}}\bm{u}_{a},\;\bm{y}^{0}= \bm{v}^{0}+\sigma(\bm{A}_{0}\xx^{0}_{0}-b_{0})$\;
}
%
\smallskip
\SetAlgoLined
%
\For{$\iter=0,1,2,\dots$}{
%
\Atdo{\texttt{\textbf{the \ac{DSO}}}}{
$\xx^{k+1}_{0}=\arg\min_{\tilde{\xx}_{0}\in\mathcal{X}_{0}}\{\inner{\nabla\phi_{0}(\xx_{0}^{k})+\bm{A}_{0}^{\top}\bm{y}^{k}}{\tilde{\xx}_{0}}+\frac{1}{2}\norm{\tilde{\xx_{0}}-\xx^{k}_{0}}^{2}_{\bm{B}_{0}}\}$
 \; \label{algorithm:PGDSO}
}
%
\Atdo{\texttt{\textbf{\ac{LA}~$\aggregator$ drawn uniformly at random}}}{
 %
\texttt{\textbf{receive \acs{DLMP}}}~$\dualk{\iter}$ \texttt{\textbf{from \ac{DSO}}} \;
\label{algorithm:PGLA}%
$\xx^{k+1}_{a}=\arg\min_{\tilde{\xx}_{a}\in\mathcal{X}_{a}}\{\inner{\nabla\phi_{a}(\xx^{k}_{a})+\bm{A}_{a}^{\top}\bm{y}^{k}}{\tilde{\xx}_{a}}+\frac{p}{2}\norm{\tilde{\xx}_{a}-\xx^{k}_{a}}^{2}_{\bm{B}_{a}}\}
$ \; 
%
$\bm{s}^{k}=\bm{A}_{a}(\xx^{k+1}_{a}-\xx^{k}_{a}) $ \;
}
%
\Atdo{\texttt{\textbf{each other aggregator~$a'\neq a$}}}{
 %
$\xx^{k+1}_{a'}=\xx^{k}_{a'} $ \;
}
%
%
 \Atdo{\texttt{\textbf{the \ac{DSO}}}}{
\texttt{\textbf{receive bid}}~$\bm{s}^{k}$ \texttt{\textbf{from \ac{LA}~$a$}} \;
%
$\bm{y}^{k+1}=\bm{y}^{k}+\sigma[\bm{A}_{0}(2\xx^{k+1}_{0}-\xx^{k}_{0})-b_{0}]+\sigma(p+1)\bm{s}^{k} $  \; \label{algorithm:dualk}
%
$\bm{v}^{k+1}=\bm{v}^{k}+\sigma\bm{s}^{k} $  \;
}
}
 \end{algorithm}
\noindent
From a practical point of view, it is important to point out that, while executing \ac{PPDLMP}, the bus-specific data (like cost function, power profiles, etc.) remain private information. This applies equally to the \ac{DSO} and the \ac{LA}. Coordination of the system-wide behaviour is achieved via exchanging information about \emph{dual variables} only, describing the DLMPs and the expressed bids of the \ac{LA}s. 

\subsection{Convergence of \ac{PPDLMP}}
We deduce the convergence of \ac{PPDLMP} from the analysis of Algorithm \ref{alg:RBCD}. In order to recover the OPF problem, we identify each function $\phi_{i}$ with a cost function of the \ac{DSO} or \ac{LA}, and $r_{i}$ is an indicator function of the feasible set $\XLA{\aggregator}$ and $\XDSO$, respectively. The sampling technology is an example of a uniform sampling of pairs involving the DSO and one aggregator. Specifically, the set of realisations of the sampling is 
$\Sigma=\left\{\{0,a\}:\aggregator\in\{1,\ldots,\cardinality{\Aggregator}\}\equiv\{1,\ldots,\nbaggregators\}\right\}.$ Each pair is realised with the uniform probability $(\Pi)_{0,a}=\frac{1}{p}$. The marginal distribution of the sampling technology is thus given by $\pi_{0}=1$ and $\pi_{a}=1/p$ for $a\in\{1,\ldots,p\}$. The optimisation problem is characterised by $d=p+1$ blocks and weighting matrix $\bm{P}=\blkdiag[\bm{I}_{m_{0}};\ldots; p\bm{I}_{m_{d}}]$, where $m_{0}$ is the dimension of the feasible set of the \ac{DSO}, and $m_{a}$ is the dimension of the feasible set of aggregator $a\in\Aggregator$. Now, define $R(\xx)=r_{0}(\xx_{0})+\sum_{a\in\mathcal{A}}r_{a}(\xx_{a})$, where $r_{0}(\xx_{0})=\delta_{\mathcal{X}_{0}}(\xx_{0})$ and $r_{a}(\xx_{a})=\delta_{\mathcal{X}_{a}}(\xx_{a})$, in which $\delta_{C}$ is the indicator function of a set $C$. We now show that Algorithm \ref{alg:PPDLMP} is a special instantiation of Algorithm \ref{alg:PDA}, which in turn we know to be equivalent to Algorithm \ref{alg:RBCD}.\\
\noindent
If the load aggregator $a$ is chosen at step~$\iter$, Line 5 in Algorithm~\ref{alg:PDA} becomes
\begin{align*}
&\bm{y}^{k+1}=\bm{y}^{k} +  \sigma \bm{A}\bm{P}(\xx^{k+1}-\xx^{k})  + \sigma\bm{u}^{k+1}\\
&=\bm{y}^{k} +  \sigma \bm{A}_{0} (\xx_{0}^{k+1} -\xx^{k}_{0})+  \sigma p\bm{A}_{a}(\xx^{k+1}_{a}-\xx^{k}_{a})+ \sigma \bm{u}^{k+1} \\
&=\bm{y}^{k}+\sigma[\bm{A}_{0}(2\xx^{k+1}_{0}-\xx^{k}_{0})-b_{0}]+\sigma p\bm{A}_{a}(\xx^{k+1}_{a}-\xx^{k}_{a})+\bm{v}^{k+1}
\end{align*}
where we have used the identity $\bm{u}^{k+1}=\bm{A}\xx^{k}-\bm{b}$ and define $\bm{v}^{k}=\sigma\bm{u}^{k}-\sigma(\bm{A}_{0}\xx^{k}_{0}-b_{0})=\sigma\sum_{a\in\mathcal{A}}(\bm{A}_{a}\xx^{k}_{a}-b_{a})$. Exploiting Line 4 in Algorithm~\ref{alg:PDA}, we find that $\bm{v}^{k}$ can be computed locally and inductively by choosing the initial condition $\bm{v}^{0}=\sigma\sum_{a\in\Aggregator}(\bm{A}_{a}\xx^{0}_{a}-b_{a})$ initially, and performing sequential updates 
$$
\bm{v}^{k+1}=\bm{v}^{k}+\sigma\bm{A}_{a}(\xx^{k+1}_{a}-\xx^{k}_{a})=\bm{v}^{k}+\sigma\bm{s}^{k}
$$
where $\bm{s}^{k}=\bm{A}_{a}(\xx^{k+1}_{a}-\xx^{k}_{a})$. This shows
$$
\bm{y}^{k+1}=\bm{y}^{k}+\sigma[\bm{A}_{0}(2\xx^{k+1}_{0}-\xx^{k}_{0})-b_{0}]+\sigma(p+1)\bm{s}^{k}+\bm{v}^{k},
$$
which agrees with Lines 4,7 and 8 of Algorithm \ref{alg:PPDLMP}. 

\subsection{Numerical Results}
\label{sec:Numerics}
We apply Algorithm \ref{alg:PPDLMP} to a realistic 15-bus network example based on the instance proposed in \cite{papavasiliou18}, over a time horizon $\mathcal{T} = \{0, 1\}$. The parameters $(R_n,X_n,S_n,B_n,V_n)$ are those used in~\cite{papavasiliou18}.  We consider variable, flexible active and reactive loads (as opposed to fixed ones, as in~\cite{papavasiliou18}): parameters  $(\lbapcn{\node},\ubapcn{\node},E_n,\tauC_n)_n$  are generated  based on the values of  \cite{papavasiliou18}.
The underlying parameter values are specified in  \Cref{tab:network14param}: line parameters $R_n,X_n,S_n,B_n$, $n\in\N$ are taken from  \cite{papavasiliou18}, $(\lbapcn{\node},\ubapcn{\node},E_n,\tauC_n)_n$ for the \textit{flexible} loads are generated based on the \emph{fixed} load values of \cite{papavasiliou18}.
%
As in \cite{papavasiliou18},  bus $11$ is the only bus to offer renewable production, with $ \bm{\upRE}_{11} \eqd [0.438,0.201]$ and $\rlbRE = \rubRE = 0$ (the renewable production is fully active). 
The bounds $(\uV_n,\oV_n)$ are set to 0.81 and $1.21$ for each  $n\in\N$, while $V_0=1.0$.

We consider a zero cost function for each \ac{LA} ($\phi_a=0$ for each $a\in\A$), meaning that \acp{LA} are indifferent to  consumption profiles for as long as their feasibility constraints are satisfied.
This is a reasonable assumption in practice for certain types of consumption flexibilities (electric vehicles, batteries).
We consider the \ac{DSO} objective
$ \phi(\xx)= \phi_0(\xx_0)=  \sum_{t\in\T}c_t( \pRE_{0t}) + k^\text{loss} \sum_{n,t} R_n \ell_{nt},$ 
with loss penalisation factor $k^\text{loss}=0.001$ and with:
$c_0:p \mapsto 2p + p^2 , \ c_1 : p \mapsto p , $
giving an \textit{expensive} time period and a \textit{cheap} one, which can be interpreted as peak and offpeak periods.
\newcommand{\hhh}{\hspace{-3pt}}
\begin{table}[!ht]
 \centering
\begin{scriptsize}
\setlength{\tabcolsep}{1.5pt}
\begin{tabular}{ccccccccc}
\hline
 $n$ &      $S_{n}$ &      $R_{n}{\cdot}10^3$ &      $X_{n}{\cdot}10^3$ &   $B_{n}{\cdot}10^3$ &       $\uP_{n}$ &               $\oP_{n}$ &     $ E_{n} $&    $\tau^{c}_{n}$ \\
\hline
  1 &  2.000 &     1.0 &   120.0 &     1.1 &   [0.593, 0.256] &    [1.566, 1.539] &  2.213 &  0.234 \\
  2 &  0.256 &    88.3 &   126.2 &     2.8 &       [0.000, 0.000] &        [0.000, 0.000] &  0.000 &  0.000 \\
  3 &  0.256 &   138.4 &   197.8 &     2.4 &   [0.003, 0.011] &     [0.020, 0.035] &  0.047 &  0.418 \\
  4 &  0.256 &    19.1 &    27.3 &     0.4 &   [0.015, 0.013] &    [0.027, 0.019] &  0.033 &  0.249 \\
  5 &  0.256 &    17.5 &    25.1 &     0.8 &   [0.021, 0.024] &    [0.043, 0.053] &  0.072 &  0.251 \\
  6 &  0.256 &    48.2 &    68.9 &     0.6 &   [0.017, 0.001] &    [0.032, 0.037] &  0.039 &  0.251 \\
  8 &  0.256 &    40.7 &    58.2 &     1.2 &   [0.021, 0.009] &     [0.040, 0.039] &  0.049 &  0.251 \\
  7 &  0.256 &    52.3 &    74.7 &     0.6 &  [-0.233,-0.210] &  [-0.173,-0.115] & -0.352 &  0.000 \\
  9 &  0.256 &    10.0 &    14.3 &     0.4 &   [0.008, 0.002] &    [0.032, 0.028] &  0.015 &  0.620 \\
 10 &  0.256 &    24.1 &    34.5 &     0.4 &   [0.004, 0.001] &     [0.024, 0.040] &  0.013 &  0.300 \\
 11 &  0.256 &    10.3 &    14.8 &     0.1 &     [0.010, 0.010] &    [0.015, 0.024] &  0.028 &  0.250 \\
 12 &  0.600 &     1.0 &   120.0 &     0.1 &   [0.243, 0.057] &    [0.642, 0.625] &  0.895 &  0.208 \\
 13 &  0.204 &   155.9 &   111.9 &     0.2 &     [0.001, 0.000] &    [0.003, 0.003] &  0.003 &  0.571 \\
 14 &  0.204 &    95.3 &    68.4 &     0.1 &   [0.015, 0.012] &    [0.032, 0.042] &  0.042 &  0.371 \\
\hline
\end{tabular}
\end{scriptsize}
\caption{Parameters for the 15 buses network based on  \cite{papavasiliou18}}
\label{tab:network14param}
\end{table}
%
%
%

The solution obtained by Algorithm~\ref{alg:PPDLMP} after 2000 iterations is illustrated in  \Cref{fig:resFlows}, which displays the active flows directions as well as the \ac{DLMP} values.
 
\begin{figure}[h]
\begin{subfigure}{0.49\columnwidth} \centering
  \begin{scriptsize}
\begin{tikzpicture}[scale=0.3]
\node [draw,circle] (node0) at (0,0) {0} ;
\node [draw,circle] (node1) at (0.0, -3.5) {1} ;
 \draw [-latex',green,line width=2pt] (node0) -- (node1); 
  
\node [left= 0.1 cm of node1] (res0) {\pbox{2.5cm}{
$3.721$ \\ 
$0.012$ \\ 
}} ;
\node [draw,circle] (node2) at (0.0, -7.0) {2} ;
 \draw [latex'-, red,line width=2pt] (node1) -- (node2); 
  
\node [left= 0.1 cm of node2] (res0) {\pbox{2.5cm}{
$3.603$ \\ 
$0.048$ \\ 
}} ;
\node [draw,circle] (node3) at (0.0, -10.5) {3} ;
 \draw [latex'-, red,line width=2pt] (node2) -- (node3); 
  
\node [left= 0.1 cm of node3] (res0) {\pbox{2.5cm}{
$3.421$ \\ 
$0.092$ \\ 
}} ;
\node [draw,circle] (node4) at (0.0, -14.0) {4} ;
 \draw [-latex',green,line width=2pt] (node3) -- (node4); 
  
\node [left= 0.1 cm of node4] (res0) {\pbox{2.5cm}{
$3.429$ \\ 
$0.093$ \\ 
}} ;
\node [draw,circle] (node5) at (0.0, -17.5) {5} ;
 \draw [-latex',green,line width=2pt] (node4) -- (node5); 
  
\node [left= 0.1 cm of node5] (res0) {\pbox{2.5cm}{
$3.433$ \\ 
$0.094$ \\ 
}} ;
\node [draw,circle] (node6) at (0.0, -21.0) {6} ;
 \draw [-latex',green,line width=2pt] (node5) -- (node6); 
  
\node [left= 0.1 cm of node6] (res0) {\pbox{2.5cm}{
$3.439$ \\ 
$0.096$ \\ 
}} ;
\node [draw,circle] (node8) at (3.5, -14.0) {8} ;
 \draw [latex'-, red, style = dashed ,line width=2pt] (node3) -- (node8); 
  
\node [right= 0.1 cm of node8] (res0) {\pbox{2.5cm}{
$0.004$ \\ 
$0.279$ \\ 
}} ;
\node [draw,circle] (node7) at (3.5, -17.5) {7} ;
 \draw [latex'-, red,line width=2pt] (node8) -- (node7); 
  
\node [below= 0.1 cm of node7] (res0) {\pbox{2.5cm}{
$-0.003$ \\ 
$0.279$ \\ 
}} ;
\node [draw,circle] (node9) at (7.0, -17.5) {9} ;
 \draw [latex'-, red,line width=2pt] (node8) -- (node9); 
  
\node [right= 0.1 cm of node9] (res0) {\pbox{2.5cm}{
$0.003$ \\ 
$0.279$ \\ 
}} ;
\node [draw,circle] (node10) at (7.0, -21.0) {10} ;
 \draw [latex'-, red,line width=2pt] (node9) -- (node10); 
  
\node [right= 0.1 cm of node10] (res0) {\pbox{2.5cm}{
$0.001$ \\ 
$0.279$ \\ 
}} ;
\node [draw,circle] (node11) at (7.0, -24.5) {11} ;
 \draw [latex'-, red,line width=2pt] (node10) -- (node11); 
  
\node [right= 0.1 cm of node11] (res0) {\pbox{2.5cm}{
$-0.0$ \\ 
$0.279$ \\ 
}} ;
\node [draw,circle] (node12) at (3.5, -3.5) {12} ;
 \draw [-latex',green,line width=2pt] (node0) -- (node12); 
  
\node [right= 0.1 cm of node12] (res0) {\pbox{2.5cm}{
$3.719$ \\ 
$0.001$ \\ 
}} ;
\node [draw,circle] (node13) at (3.5, -7.0) {13} ;
 \draw [-latex',green,line width=2pt] (node12) -- (node13); 
  
\node [right= 0.1 cm of node13] (res0) {\pbox{2.5cm}{
$3.758$ \\ 
$0.016$ \\ 
}} ;
\node [draw,circle] (node14) at (3.5, -10.5) {14} ;
 \draw [-latex',green,line width=2pt] (node13) -- (node14); 
  
\node [right= 0.1 cm of node14] (res0) {\pbox{2.5cm}{
$3.781$ \\ 
$0.024$ \\ 
}} ;
\end{tikzpicture}

\end{scriptsize}
 \subcaption{ t=0}
\end{subfigure}
\begin{subfigure}{0.49\columnwidth} \centering
\centering
 \begin{scriptsize}
\begin{tikzpicture}[scale=0.3]
\node [draw,circle] (node0) at (0,0) {0} ;
\node [draw,circle] (node1) at (0.0, -3.5) {1} ;
 \draw [-latex',green,line width=2pt] (node0) -- (node1); 
  
\node [left= 0.1 cm of node1] (res0) {\pbox{2.5cm}{
$1.004$ \\ 
$0.005$ \\ 
}} ;
\node [draw,circle] (node2) at (0.0, -7.0) {2} ;
 \draw [latex'-, red,line width=2pt] (node1) -- (node2); 
  
\node [left= 0.1 cm of node2] (res0) {\pbox{2.5cm}{
$0.98$ \\ 
$0.02$ \\ 
}} ;
\node [draw,circle] (node3) at (0.0, -10.5) {3} ;
 \draw [latex'-, red,line width=2pt] (node2) -- (node3); 
  
\node [left= 0.1 cm of node3] (res0) {\pbox{2.5cm}{
$0.942$ \\ 
$0.041$ \\ 
}} ;
\node [draw,circle] (node4) at (0.0, -14.0) {4} ;
 \draw [-latex',green,line width=2pt] (node3) -- (node4); 
  
\node [left= 0.1 cm of node4] (res0) {\pbox{2.5cm}{
$0.946$ \\ 
$0.042$ \\ 
}} ;
\node [draw,circle] (node5) at (0.0, -17.5) {5} ;
 \draw [-latex',green,line width=2pt] (node4) -- (node5); 
  
\node [left= 0.1 cm of node5] (res0) {\pbox{2.5cm}{
$0.949$ \\ 
$0.043$ \\ 
}} ;
\node [draw,circle] (node6) at (0.0, -21.0) {6} ;
 \draw [-latex',green,line width=2pt] (node5) -- (node6); 
  
\node [left= 0.1 cm of node6] (res0) {\pbox{2.5cm}{
$0.951$ \\ 
$0.044$ \\ 
}} ;
\node [draw,circle] (node8) at (3.5, -14.0) {8} ;
 \draw [latex'-, red, style = dashed ,line width=2pt] (node3) -- (node8); 
  
\node [right= 0.1 cm of node8] (res0) {\pbox{2.5cm}{
$0.003$ \\ 
$0.115$ \\ 
}} ;
\node [draw,circle] (node7) at (3.5, -17.5) {7} ;
 \draw [latex'-, red,line width=2pt] (node8) -- (node7); 
  
\node [below= 0.1 cm of node7] (res0) {\pbox{2.5cm}{
$-0.0$ \\ 
$0.115$ \\ 
}} ;
\node [draw,circle] (node9) at (7.0, -17.5) {9} ;
 \draw [latex'-, red,line width=2pt] (node8) -- (node9); 
  
\node [right= 0.1 cm of node9] (res0) {\pbox{2.5cm}{
$0.002$ \\ 
$0.115$ \\ 
}} ;
\node [draw,circle] (node10) at (7.0, -21.0) {10} ;
 \draw [latex'-, red,line width=2pt] (node9) -- (node10); 
  
\node [right= 0.1 cm of node10] (res0) {\pbox{2.5cm}{
$0.001$ \\ 
$0.115$ \\ 
}} ;
\node [draw,circle] (node11) at (7.0, -24.5) {11} ;
 \draw [latex'-, red,line width=2pt] (node10) -- (node11); 
  
\node [right= 0.1 cm of node11] (res0) {\pbox{2.5cm}{
$-0.0$ \\ 
$0.115$ \\ 
}} ;
\node [draw,circle] (node12) at (3.5, -3.5) {12} ;
 \draw [-latex',green, style = dashed ,line width=2pt] (node0) -- (node12); 
  
\node [right= 0.1 cm of node12] (res0) {\pbox{2.5cm}{
$3.567$ \\ 
$0.732$ \\ 
}} ;
\node [draw,circle] (node13) at (3.5, -7.0) {13} ;
 \draw [-latex',green,line width=2pt] (node12) -- (node13); 
  
\node [right= 0.1 cm of node13] (res0) {\pbox{2.5cm}{
$3.583$ \\ 
$0.737$ \\ 
}} ;
\node [draw,circle] (node14) at (3.5, -10.5) {14} ;
 \draw [-latex',green,line width=2pt] (node13) -- (node14); 
  
\node [right= 0.1 cm of node14] (res0) {\pbox{2.5cm}{
$3.593$ \\ 
$0.741$ \\ 
}} ;
\end{tikzpicture}

\end{scriptsize}
\subcaption{ t=1 }
\end{subfigure}%
 \caption{Directions of  active flows $\bm{f}$  and \acp{DLMP} $(y^p,y^q )$ at the solution given by Algorithm \protect\ref{alg:PPDLMP}. Saturated lines are dashed.}
 \label{fig:resFlows}
 \end{figure}

The solutions show that the active (and reactive) \acp{DLMP} obtained for each time period are close to the \acp{DLMP} at the root node $(\dlmpn0, \dlmqn0)$, with the following exceptions:
\begin{itemize}
\item For the branch composed of nodes $8,7,9,10,11$,  active \acp{DLMP} are close to~$0.0$ 
  due to the presence of renewable production (at null cost) at node 11, and of negative load at node $7$, which together fully compensate for the demand on this branch.
  Since Line $(3,8)$ is saturated, no energy can be exported further.
\item  Active \acp{DLMP} on the branch composed of nodes $(12,13,14)$ at  $t=1$ are much larger than on other nodes: this is explained by the congestion  of line $(0,12)$.
\item  The \ac{DLMP} for node 7 and $t=0$ is \emph{strictly negative}: the (negative) consumption for this node is at its upper bound $p_{7,0}=\overline{P}_{7,0}=-0.173$.  The negative \ac{DLMP} suggests that the system will be better off if less power is {injected} by node 7. 
\end{itemize}

\begin{figure}[!ht]
\centering
\includegraphics[width=0.7\columnwidth]{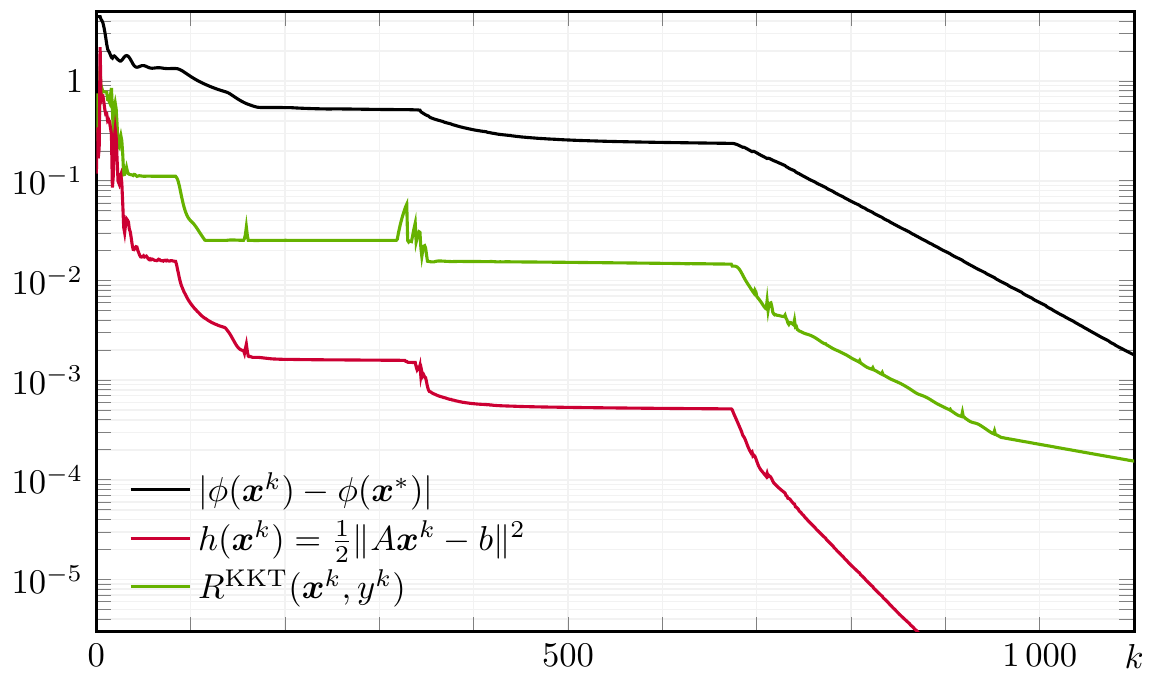}
 \caption{Convergence of \acp{DLMP} $\dualk{\iter}$ and the objective function value $\Phi(\xx^{k})$%
}
  \label{fig:lastiterate}  
\end{figure} 
\begin{figure}[!ht]
\centering
\includegraphics[width=0.7\columnwidth]{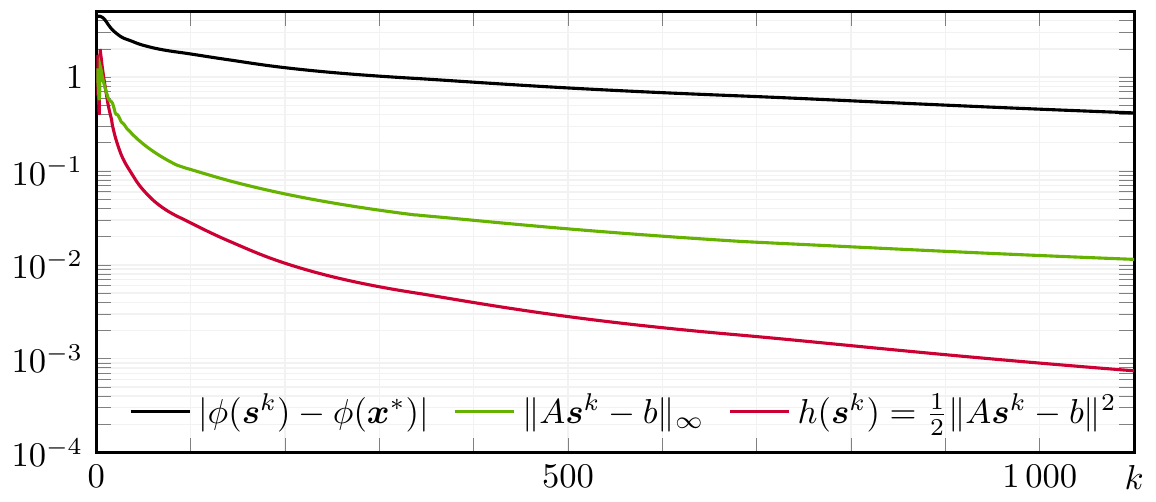}
\\\includegraphics[width=0.7\columnwidth]{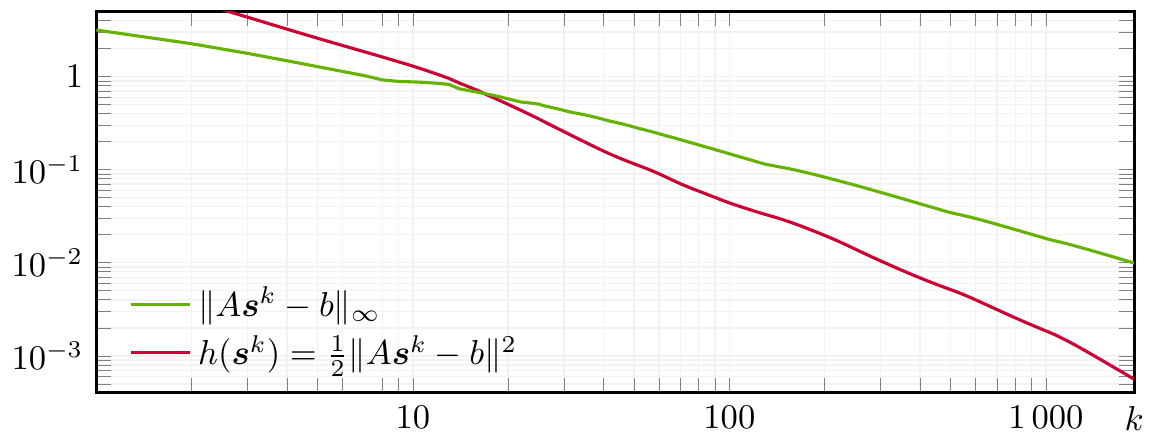}
 \caption{Convergence of the average iterate~$\bm{w}^{k}$} 
 \label{fig:ergodicaverage} 
\end{figure}

Convergence of Algorithm~\ref{alg:PPDLMP} for the 15-bus network is shown in \Cref{fig:lastiterate,fig:ergodicaverage}.
\Cref{fig:lastiterate} displays 
the convergence of the last iterate with respect to various criteria: 
\begin{itemize}
\item[(i)] convergence of~$\cost{\xallk{\iter}}$  to the optimal cost~$\cost{\xallsol}$, where we computed $\xx^{\ast}$ using the CvxOpt Python library; 
\item[(ii)] convergence to zero of the primal residuals~$\constraint{\xallk{\iter}}$ and~$\norm{\Aall\xallk{\iter}-\bvector}$;
\item[(ii)] convergence to zero of the stopping criterion formed on the KKT residual $R^{\text{KKT}}(\xx^{k},y^{k})=\text{dist}_{\infty}(0,(\partial_{x},-\partial_{y})L(\xx^{k},y^{k}))$, 
where $\lagrangian{\xall}{\dualvariable}=\extcost{\xall}+\inner{\dualvariable}{\Aall\xall-\bvector}$ denotes the Lagrangian of~\eqref{eq:P} in the feasible case; 
\item[(iv)] convergence of the \Cref{fig:ergodicaverage} shows the convergence to zero of the primal infeasiblity in the averaged trajectory $\bm{w}^{k}$, as predicted by the theory. 
\end{itemize}

\section{Conclusion}
In this paper we developed a novel random block-coordinate descent method for ill-posed convex non-smooth optimisation problems. Our scheme gives optimal iteration complexity guarantees in terms of the last iterate of the sequence generated by the algorithm. As such, we directly generalise results from \cite{luke18} to a fully distributed optimisation setting. Motivated by the need to achieve a distributed optimisation of the power system, we follow a data-driven approach leading to a potentially inconsistent AC-OPF formulation. We show that our Algorithmic scheme is immediately applicable to such a challenging and important setting, and achieves distributed control of the power grid via distributed locational marginal prices as price signals to independent agents (i.e. aggregators). Future work should consist in extending the method to non-convex optimisation problems, so that exact formulation of the power flow constraints can be implemented. We also plan to conduct extensions of this work where the electric network is exposed to stochastic uncertainty. 

\backmatter

\section*{Declarations}

\begin{itemize}
\item Availability of data and materials: Not Applicable
\item Funding: This research benefited from the support of the FMJH Program Gaspard Monge for optimization and operations research and their interactions with data science.
\item Competing interests: No competing interests
\item Acknowledgements: We thank Yura Malitsky for insightful discussions on this topic and Olivier Bilenne for his initial work on this paper.
\item Authors' contributions: Mathias Staudigl developed the algorithm and the analysis. Paulin Jacquot contributed to the development of the algorithm, and implemented the numerical example. 
\end{itemize}

\begin{appendices}
\section{Technical Appendix}

\subsection{General Facts}
\label{app:auxiliary}
The following relations will be useful in the analysis. 

\begin{lemma}[\cite{BauCom16}, Corollary 2.14]
\label{lem:BauCom}
For all $\xx,\bm{y}\in\R^{d}$ and $t\in[0,1]$, one has 
$$
\norm{ t\xx+(1-t)\bm{y}}^{2}=t\norm{\xx}^{2}+(1-t)\norm{\bm{y}}^{2}-t(1-t)\norm{\xx-\bm{y}}^{2}.
$$
\end{lemma}

Let $(\Omega,\filtrationfunction,(\filtrationk k)_{k\geq0},\Prob)$ be a filtered probability space satisfying the usual conditions. We need the celebrated Robbins-Siegmund Lemma for the convergence analysis \cite{RS71}. 
\begin{lemma}\label{lem:RS}
For every $k\geq 0$, let $\theta_{k},u_{k},\zeta_{k}$ and $t_{k}$ be non-negative $\filtrationk k$-measurable random variables such that $(\zeta_{k})_{n\geq0}$ and $(t_{k})_{k\geq 0}$ are summable and for all $k\geq 0$,
\begin{equation}
\Ex[\theta_{k+1}\vert \filtrationk k]\leq(1+t_{k})\theta_{k}+\zeta_{k}-u_{k}\qquad\Prob-\text{a.s.}
\end{equation}
Then $(\theta_{k})_{k\geq 0}$ converges and $(u_{k})_{k\geq 0}$ is summable $\Prob$-a.s. 
\end{lemma}

\subsection{Connections among the iterates produced by Algorithms \ref{alg:RBCD} and \ref{alg:PDA}}
 
To obtain a compact representation of the primal updates, let us introduce the iid Bernoulli process $\epsilon_{i}^{k}:\Omega\to\{0,1\}$ by 
\[
\epsilon_{i}^{k}(\omega)=\left\{\begin{array}{cc} 
1 & \text{if }i\in\iota_{k}(\omega)\\
0 & \text{else}
\end{array}\right.
\]
and the random matrix $\bm{E}_{k}=\blkdiag[\epsilon_{1}^{k}\bm{I}_{m_{1}},\ldots,\epsilon^{k}_{d}\bm{I}_{m_{d}}]\in\{0,1\}^{m\times m}$. This matrix corresponds to the identity for those blocks that are activated in round $k$, and zero for those not activated. Therefore, we call it the \emph{activation matrix} of the Algorithm. By definition of the sampling process, we have $\Ex(\epsilon^{k}_{i})=\Prob(i\in\iota_{k})=\pi_{i}$ for all $i\in[d]$ and $k\geq 0$, so that $\Ex[\bm{E}_{k}]=\bm{P}^{-1}$.

Let $\epsilon^{k}:=(\epsilon_{1}^{k},\ldots,\epsilon_{d}^{k})\in\{0,1\}^{d}$, and $\mathcal{F}_{k}:=\sigma(\xx^{0},\epsilon^{0},\ldots,\epsilon^{k-1})$ denote the history of the process up to step $k$. We denote by $\Ex_{k}[\cdot]:=\Ex[\cdot\vert\filtrationfunction_{k}]$ the conditional expectation at stage $k$. Denoting by $\hat{\xx}^{k+1}=\FB^{k}(\xx^{k})$, we obtain the compact representation of the primal update as 
\begin{equation}\label{eq:xupdate}
\xx^{k+1}=\xx^{k}+\bm{E}_{k}(\hat{\xx}^{k+1}-\xx^{k}).
\end{equation}
Furthermore, $\Ex_{k}(\xx^{k+1})=\bm{P}^{-1}\hat{\xx}^{k+1}+(\bm{I}-\bm{P}^{-1})\xx^{k},$ so that 
\begin{equation}\label{eq:Exhatx}
\Ex_{k}[\xx^{k}+\bm{P}(\xx^{k+1}-\xx^{k})]=\hat{\xx}^{k+1}.
\end{equation}

\begin{lemma}\label{lem:Ex1}
For all $k\geq 0$, we have 
\begin{equation}\label{eq:Ex1}
\begin{split}
\frac{1}{2}\norm{\bm{A}(\hat{\xx}^{k+1}-\bm{w}^{k})}^{2}&=\frac{1}{2}\Ex_{k}\left[\norm{\bm{A}(\xx^{k}+\bm{P}(\xx^{k+1}-\xx^{k})-\bm{w}^{k})}^{2}\right]\\
&-\frac{1}{2}\Ex_{k}\left[\norm{\bm{A}(\bm{P}\bm{E}_{k}-\bm{I})(\hat{\xx}^{k+1}-\xx^{k})}^{2}\right].
\end{split}
\end{equation}
\end{lemma}
\begin{proof}
Note that for any random variable $X:\Omega\to\R^{q}$, it holds true $\norm{\Ex[X]}^{2}=\Ex\left[\norm{X}^{2}\right]-\Ex\left[\norm{X-\Ex X}^{2}\right]$. Set $X:=\bm{A}(\xx^{k}+\bm{P}(\xx^{k+1}-\xx^{k})-\bm{w}^{k})$. Then, by eq. \eqref{eq:Exhatx}, we see $\Ex_{k}[X]=\bm{A}(\hat{\xx}^{k+1}-\bm{w}^{k})$. Furthermore, 
\begin{equation}\label{eq:centerd}
X-\Ex_{k}[X]=\bm{A}\left((\bm{I}-\bm{P})\xx^{k}+\bm{P}\xx^{k+1}-\hat{\xx}^{k+1}\right)=\bm{A}(\xx^{k}+\bm{P}(\xx^{k+1}-\xx^{k})-\hat{\xx}^{k+1}).
\end{equation}
Using \eqref{eq:xupdate}, we get
$$
\bm{P}\xx^{k+1}-\hat{\xx}^{k+1}=(\bm{P}\bm{E}_{k}-\bm{I})\hat{\xx}^{k+1}+\bm{P}(\bm{I}-\bm{E}_{k})\xx^{k},
$$
which implies 
\begin{equation}\label{eq:centerd2}
\bm{A}(\xx^{k}+\bm{P}(\xx^{k+1}-\xx^{k})-\hat{\xx}^{k+1})=X-\Ex_{k}[X]=\bm{A}(\bm{P}\bm{E}_{k}-\bm{I})(\hat{\xx}^{k+1}-\xx^{k}).
\end{equation}
Collecting all therms above, one obtains 
\begin{align*}
\frac{1}{2}\Ex_{k}\left[\norm{\bm{A}(\xx^{k}+\bm{P}(\xx^{k+1}-\xx^{k})-\bm{w}^{k})}^{2}\right]&=\frac{1}{2}\norm{\bm{A}(\hat{\xx}^{k+1}-\bm{w}^{k})}^{2}\\
&+\frac{1}{2}\Ex_{k}\left[\norm{\bm{A}(\bm{P}\bm{E}_{k}-\bm{I})(\hat{\xx}^{k+1}-\xx^{k})}^{2}\right].
\end{align*}
\end{proof}

The next Lemma provides a compact expression for the correction term in eq. \eqref{eq:Ex1}.
\begin{lemma}\label{lem:Ex2}
Let $\bm{\Xi}=\Ex[\bm{E}_{k}\bm{P}\bm{A}^{\top}\bm{A}\bm{P}\bm{E}_{k}]$, which is a symmetric time-independent matrix, thanks to iid sampling. We have 
\begin{equation}\label{eq:Ex2}
\Ex_{k}\left[\norm{\bm{A}(\bm{P}\bm{E}_{k}-\bm{I})(\hat{\xx}^{k+1}-\xx^{k})}^{2}\right]=\norm{\hat{\xx}^{k+1}-\xx^{k}}^{2}_{\Xi-\bm{A}^{\top}\bm{A}}.
\end{equation}
\end{lemma}
\begin{proof}
By definition of the sampling procedure, we have $\Ex_{k}[\bm{P}\bm{E}_{k}]=\bm{I}_{m}$. Some simple algebra shows then that  
\begin{align*}
\Ex_{k}\left[\norm{\bm{A}(\bm{P}\bm{E}_{k}-\bm{I})(\hat{\xx}^{k+1}-\xx^{k})}^{2}\right]&=\Ex_{k}\left[(\hat{\xx}^{k+1}-\xx^{k})^{\top}\bm{M}_{k}(\hat{\xx}^{k+1}-\xx^{k})\right],
\end{align*}
where $\bm{M}_{k}:=\bm{E}_{k}\bm{P}\bm{A}^{\top}\bm{A}\bm{P}\bm{E}_{k}-\bm{A}^{\top}\bm{A}\bm{P}\bm{E}_{k}-\bm{E}_{k}\bm{P}\bm{A}^{\top}\bm{A}+\bm{A}^{\top}\bm{A}$ is a symmetric matrix, satisfying $\Ex_{k}\bm{M}_{k}=\bm{\Xi}-\bm{A}^{\top}\bm{A}$. Hence, 
\begin{align*}
\Ex_{k}\left[(\hat{\xx}^{k+1}-\xx^{k})^{\top}\bm{M}_{k}(\hat{\xx}^{k+1}-\xx^{k})\right]&=\Ex_{k}\left[\tr\left(\bm{M}_{k}(\hat{\xx}^{k+1}-\xx^{k})(\hat{\xx}^{k+1}-\xx^{k})^{\top}\right)\right]\\
&=\tr\left(\Ex_{k}[\bm{M}_{k}](\hat{\xx}^{k+1}-\xx^{k})(\hat{\xx}^{k+1}-\xx^{k})^{\top}\right)\\
&=\tr\left((\Xi-\bm{A}^{\top}\bm{A})(\hat{\xx}^{k+1}-\xx^{k})(\hat{\xx}^{k+1}-\xx^{k})^{\top}\right)\\
&=\norm{\hat{\xx}^{k+1}-\xx^{k}}_{\bm{\Xi}-\bm{A}^{\top}\bm{A}}.
\end{align*}
\end{proof}
\begin{remark}
The matrix $\bm{\Xi}$ can be given a simple expression in terms of the probability matrix $\bm{\Pi}$. A direct computation shows that $\bm{\Xi}$ is a block-symmetric matrix with $d^{2}$ blocks $\bm{\Xi}[i,j],1\leq i,j\leq d,$ each block having the form 
$$
\bm{\Xi}[i,j]=\frac{(\bm{\Pi})_{ij}}{\pi_{i}\pi_{j}}\bm{A}^{\top}_{i}\bm{A}_{j}^{\top}\text{ if }i\neq j,\text{ and }\bm{\Xi}[i,i]=\frac{1}{\pi_{i}}\bm{A}^{\top}_{i}\bm{A}_{i}\quad \forall i\in[d].
$$
In special instances this matrix can be simplified even further:
\begin{enumerate}
\item If $\Prob(\abs{\mathcal{I}}=1)=1$ the random sampling is a single coordinate activation mechanism. In this case the associated matrix $\bm{\Pi}$ is diagonal with entries $\pi_{i}$. It follows that $\bm{\Xi}[i,i]=\frac{1}{\pi_{i}}\bm{A}_{i}^{\top}\bm{A}_{i}$ for all $i\in[d]$ and $\bm{\Xi}[i,j]=0$ for $i\neq j$. 
\item We say that the matrix $\bm{A}$ has \emph{orthogonal design} if $\bm{A}_{i}^{\top}\bm{A}_{j}$ is the zero matrix for $i\neq j$. Also in this case the matrix $\bm{\Xi}$ is block diagonal, with the same coordinate expression as above. 
\end{enumerate}
\end{remark}

%
 
 \subsection{Properties of the penalty function $h$}
 We collect some important identities involving the penalty function $h$ in this subsection. From the definition of the iterate $\bm{z}^{k}$, we see 
 \begin{equation}\label{eq:hzk}
 \begin{split}
 \nabla h(\bm{z}^{k})&=\bm{A}^{\top}(\bm{A}\bm{z}^{k}-\bm{b})=\frac{1}{S_{k}}\left(S_{k-1}\bm{A}^{\top}(\bm{A}\bm{w}^{k}-\bm{b})+\sigma_{k}\bm{A}^{\top}(\bm{A}\xx^{k}-\bm{b})\right)\\
 &=(1-\theta_{k})\nabla h(\bm{w}^{k})+\theta_{k}\nabla h(\xx^{k}).
 \end{split}
 \end{equation}
 
 Furthermore, we note that the definition of the iterate $\bm{w}^{k+1}$ implies 
 \begin{equation}\label{eq:hwk}
 \begin{split}
 h(\bm{w}^{k+1})&=h\left((1-\theta_{k})\bm{w}^{k}+\theta_{k}(\xx^{k}+\bm{P}(\xx^{k+1}-\xx^{k}))\right)\\
 &\overset{\eqref{eq:hstrongconvex}}{=}(1-\theta_{k})h(\bm{w}^{k})+\theta_{k}h(\xx^{k}+\bm{P}(\xx^{k+1}-\xx^{k}))\\
 &-\frac{\theta_{k}(1-\theta_{k})}{2}\norm{\bm{A}(\xx^{k}+\bm{P}(\xx^{k+1}-\xx^{k})-\bm{w}^{k})}^{2}.
 \end{split}
 \end{equation}
 
 \begin{lemma}
 We have 
 \begin{equation}\label{eq:hforward}
 \Ex_{k}[h(\xx^{k}+\bm{P}(\xx^{k+1}-\xx^{k}))]=h(\hat{\xx}^{k+1})+\frac{1}{2}\Ex_{k}[\norm{\bm{A}(\xx^{k}+\bm{P}(\xx^{k+1}-\xx^{k}))-\hat{\xx}^{k+1}}^{2}].
 \end{equation}
 \end{lemma}
 \begin{proof}
The Pythagorean identity gives 
 \begin{align*}
 h(\xx^{k}+\bm{P}(\xx^{k+1}-\xx^{k}))&=\frac{1}{2}\norm{(\bm{A}\hat{\xx}^{k+1}-\bm{b})+\bm{A}(\xx^{k}+\bm{P}(\xx^{k+1}-\xx^{k})-\hat{\xx}^{k+1})}^{2}\\
 &=h(\hat{\xx}^{k+1})+\frac{1}{2}\norm{\bm{A}(\xx^{k}+\bm{P}(\xx^{k+1}-\xx^{k})-\hat{\xx}^{k+1})}^{2}\\
 &+\inner{\bm{A}\hat{\xx}^{k+1}-\bm{b}}{\bm{A}(\xx^{k}+\bm{P}(\xx^{k+1}-\xx^{k})-\hat{\xx}^{k+1})}.
 \end{align*}
 Taking conditional expectations on both sides, using eq. \eqref{eq:Exhatx}, establishes the claimed identity.
 \end{proof}
 
\section{Step size policy for the accelerated algorithm}
 \label{app:parameters}
Assume that parameter coupling equation \eqref{eq:coupling} holds, where 
$$
\alpha=\lambda_{\max}(\bm{M}^{-1/2}\Xi\bm{M}^{-1/2})^{-1}\text{ and } \beta=\alpha\lambda_{\max}(\bm{M}^{-1/2}(\Lambda+\Upsilon)\bm{M}^{-1/2}).
$$
Set $\kappa=\frac{\beta}{\alpha}$. With the choice $\bm{B}=\bm{P}^{-2}\Upsilon$, the matrix inequality \eqref{eq:MI2} reads as
$$
[(\alpha-\beta\tau_{k})(\bm{I}+\tau_{k}\bm{P})-\beta\tau^{2}_{k}(\bm{I}-\bm{P})]\tau_{k+1}^{2}+\tau_{k}^{2}[\beta\bm{I}+\alpha(\bm{I}-\bm{P})]\tau_{k+1}-\alpha\tau_{k}^{2}\bm{I}\succeq 0.
$$
Exploiting the block structure of the involved matrices, we can equivalently write this condition as a system of quadratic inequalities 
\begin{equation}\label{eq:tauappendix}
\tau^{2}_{k+1}\left[(\alpha-\beta\tau_{k})(\pi_{i}+\tau_{k})+\beta(1-\pi_{i})\tau^{2}_{k}\right]-\tau_{k+1}\tau^{2}_{k}[\alpha(1-\pi_{i})-\beta\pi_{i}]-\tau^{2}_{k}\alpha\pi_{i}\geq 0,
\end{equation}
holding for all $i\in[d]$ simultaneously. This defines a quadratic inequality in $\tau_{k+1}$ of the form $c_{i,1}(\tau_{k})\tau^{2}_{k+1}+c_{i,2}(\tau_{k})\tau_{k+1}-c_{i,3}(\tau_{k})\geq 0$, with coefficients 
\begin{align*}
&c_{i,1}(y):=(\alpha-\beta y)(\pi_{i}+y)+\beta(1-\pi_{i})y^{2}=\alpha\pi_{i}\left[1-\kappa y^{2}-y(\kappa-1/\pi_{i})\right],\\
&c_{i,2}(y):=y^{2}[\beta\pi_{i}-\alpha(1-\pi_{i})]=\alpha\pi_{i}y^{2}(\kappa-1/\pi_{i}+1),\\
&c_{i,3}(y):=y^{2}\alpha\pi_{i}
\end{align*}
The structure of the coefficients shows that we can eliminate the parameters $\alpha$ and $\pi_{i}$, and just continue with the coefficients 
\begin{align*}
a_{i,1}(y):=1-\kappa y^{2}-y(\kappa-1/\pi_{i}),\;a_{i,2}(y):=y^{2}(\kappa-1/\pi_{i}+1),\; a_{i,3}(y):=y^{2}.
\end{align*}
Define the function $F:\R^{2}\to\R$ by 
$$
F(x,y):=\max_{i\in[d]}\{a_{i,1}(y)x^2+a_{i,2}(y)x-a_{i,3}(y)\}\equiv \max_{i\in[d]}F_{i}(x,y).
$$
This gives us an implicit relation for the step size policy \eqref{eq:tau2} by $F(\tau_{k+1},\tau_{k})=0$ from which we take the positive solution, if it exists. To analyse this implicit relation, we proceed as follows. 
\begin{lemma}\label{lem:F}
The equation $F(x,0)=0$ has the unique solution $x=0$. Furthermore, the equation $F(x,\alpha/\beta)=0$ has the unique positive solution $x=\frac{\alpha}{\beta}=\frac{1}{\kappa}$. 
\end{lemma}
\begin{proof}
This is a direct numerical computation.
\end{proof}
By definition we have $\kappa=\max_{i\in[d]}\frac{L_{i}+\mu_{i}}{\mu_{i}\pi_{i}}\geq\frac{1}{\pi_{i}}>1$ for all $i\in[d]$ as well. Consequently, 
$$
\delta_{i}:=\kappa-\frac{1}{\pi_{i}}\geq 0\qquad\forall i\in[d],
$$
and we obtain the concise representation $a_{i,1}(y)=1-\kappa y^{2}-y\delta_{i}$, as well as, $a_{i,2}(y)=y^{2}(\delta+1)$. It is clear that $a_{i,1}(y)>0$ for $0<y<\frac{-\delta_{i}+\sqrt{\delta_{i}^{2}+4\kappa}}{2\kappa}$. Since $\pi_{i}\in(0,1)$, it can be readily verified that restricting $y\in(0,1/\kappa)$ is a sufficient condition for $a_{i,1}(y)>0$. Hence, if $\tau_{k}\in(0,1/\kappa)$, then the characterisation of the step size reduces to finding the root of the polynomial equation 
$$
\tau^{2}_{k+1}+\frac{a_{2}(\tau_{k})}{a_{1}(\tau_{k})}\tau_{k+1}-\frac{a_{3}(\tau_{k})}{a_{1}(\tau_{k})}=0.
$$
Since $\frac{a_{2}(y)}{a_{1}(y)}=y^{2}\frac{\delta+1}{1-\kappa y^{2}-y\delta}$, and $\frac{a_{3}(y)}{a_{1}(y)}=y^{2}\frac{1}{1-\kappa y^{2}-\delta y}$, we set 
$$
b_{k}:=\frac{\delta+1}{1-\tau_{k}^{2}\kappa-\delta \tau_{k}}\text{ provided that }\tau_{k}\in(0,1/\kappa),
$$
so that our polynomial becomes 
\begin{equation}
\tau^{2}_{k+1}+\tau^{2}_{k}b_{k}\tau_{k+1}-\tau^{2}_{k}\frac{b_{k}}{\delta+1}=0.
\end{equation}
This has the unique positive solution 
\begin{equation}
\tau_{k+1}=\frac{\tau_{k}\left[-b_{k}\tau_{k}+\sqrt{b_{k}^{2}\tau^{2}_{k}+4b_{k}/(\delta+1)}\right]}{2}=\frac{2}{\delta+1}\frac{1}{1+\sqrt{1+\frac{4}{b_{k}\tau^{2}_{k}(1+\delta)}}}.
\end{equation}
Setting $d_{k}:=1-\delta_{k}\tau_{k}-\kappa\tau^{2}_{k}$, and $s_{k}:=(1+\delta)\tau_{k}$, this recursion can be equivalently stated as 
$$
s_{k+1}=\frac{2s_{k}}{s_{k}+\sqrt{s^{2}_{k}+4d_{k}}}.
$$
To wit, we note that 
\begin{align*}
s_{k+1}=\frac{2}{1+\sqrt{1+\frac{4d_{k}}{\tau^{2}_{k}(1+\delta)^{2}}}}=\frac{2}{1+\sqrt{1+\frac{4d_{k}}{s^{2}_{k}}}}=\frac{2s_{k}}{s_{k}+\sqrt{s^{2}_{k}+4d_{k}}}.
\end{align*}
Hence, $\frac{s_{k+1}}{s_{k}}=\frac{2}{1+\sqrt{1+\frac{4d_{k}}{s^{2}_{k}}}}\leq 1$ using that $d_{k}>0$. It follows by induction that $s_{k}\leq s_{0}=(1+\delta)\tau_{0}$. We deduce that the sequence $(\tau_{k})_{k}$ is monotonically decreasing and non-negative. Therefore, $\lim_{k\to\infty}\tau_{k}=\tau_{\infty}\geq 0$ exists. We argue that $\tau_{\infty}=0$. Suppose that $\tau_{\infty}\geq t>0$. Then $s_{\infty}=\lim_{k\to\infty}s_{k}=(1+\delta)\tau_{\infty}$, and $\lim_{k\to\infty}(s_{k+1}-s_{k})=0$. But then 
$$
s_{\infty}=\frac{2s_{\infty}}{s_{\infty}+\sqrt{s^{2}_{\infty}+4d_{\infty}}}<s_{\infty},
$$
a contradiction! We conclude $\tau_{\infty}=0$. If follows that $\lim_{k\to\infty}b_{k}=1+\delta$, and $\lim_{k\to\infty}d_{k}=1$. Moreover, $d_{k}\uparrow 1$. Hence, 
$$
s_{k+1}=\frac{2s_{k}}{s_{k}+\sqrt{s^{2}_{k}+4d_{k}}}\geq \frac{2s_{k}}{s_{k}+\sqrt{s^{2}_{k}+4}}=\frac{2}{1+\sqrt{1+4/s_{k}^{2}}},
$$
which is equivalent to $\frac{1-s_{k+1}}{s^{2}_{k+1}}\leq \frac{1}{s^{2}_{k}}$. It follows from \cite{tseng08apgm} that $s_{k}\geq \frac{2s_{0}}{ks_{0}+2}$ for all $k\geq 0$, which implies 
$$
\tau_{k}\geq \frac{2\tau_{0}}{(1+\delta)\tau_{0}k+2}, \text{ provided }\tau_{0}\in(0,1/\kappa).
$$
This implies via our parameter coupling $\sigma_{k}=\alpha/\tau_{k}-\beta$ for $k\geq 1$, that 
\begin{align*}
S_{k}&=1+\sum_{i=1}^{k}\frac{\alpha}{\tau_{i}}-k\beta\leq 1+ \frac{\alpha(1+\delta)}{2}\sum_{i=1}^{k}i+k\left(\frac{\alpha}{\tau_{0}}-\beta\right)\\
&=1+\frac{\alpha(1+\delta)}{4}k(k+1)+k\sigma_{0}.
\end{align*}
Hence, $S_{k}=O(k^{2})$.

\end{appendices}


\end{document}